\documentclass[runningheads]{llncs}
\usepackage{graphicx}

\usepackage{lscape}
\usepackage{pdflscape}

\usepackage{stmaryrd}

\usepackage{cmll}
\usepackage[table]{xcolor}
\usepackage{comment}

\usepackage{mathdots}

\usepackage{amscd}
\usepackage{amssymb}
\usepackage{amsmath}
\usepackage{mathptmx}
\usepackage{mathrsfs}
\usepackage{color}
\usepackage{xspace}
\usepackage{bussproofs}
\EnableBpAbbreviations

\usepackage{tikz}

\usepackage{txfonts}

\usepackage{centernot}

\usepackage{mathtools}

\usepackage{relsize}

\usepackage{multirow}

\usepackage{multicol}

\usepackage{array}
\newcolumntype{C}[1]{>{\centering\arraybackslash}p{#1}}
\newcolumntype{L}[1]{>{\arraybackslash}p{#1}}

\usepackage{hyperref}
\usetikzlibrary{arrows}
\usetikzlibrary{matrix}
\usetikzlibrary{patterns}
\usetikzlibrary{shapes}

\newcommand{\val}[1]{[\![{#1}]\!]}
\newcommand{\descr}[1]{(\![{#1}]\!)}

\def\fCenter{{\mbox{$\ \vdash\ $}}}

\newcommand{\fns}{\footnotesize}

\def\mc{\multicolumn}

\newcommand{\xtop}{\ensuremath{\top}\xspace}
\newcommand{\xbot}{\ensuremath{\bot}\xspace}
\newcommand{\xand}{\ensuremath{\wedge}\xspace}
\newcommand{\xor}{\ensuremath{\vee}\xspace}

\newcommand{\XTOP}{\ensuremath{\hat{\top}}\xspace}
\newcommand{\XBOT}{\ensuremath{\check{\bot}}\xspace}
\newcommand{\XAND}{\ensuremath{\:\hat{\wedge}\:}\xspace}
\newcommand{\XOR}{\ensuremath{\:\check{\vee}\:}\xspace}

\newcommand{\topi}{\ensuremath{\mathrm{t}}_I\xspace}
\newcommand{\boti}{\ensuremath{\mathrm{f}}_I\xspace}
\newcommand{\topc}{\ensuremath{\mathrm{t}}_C\xspace}
\newcommand{\botc}{\ensuremath{\mathrm{f}}_C\xspace}
\newcommand{\topli}{\ensuremath{1}_I\xspace}
\newcommand{\botli}{\ensuremath{0}_I\xspace}
\newcommand{\toplc}{\ensuremath{1}_C\xspace}
\newcommand{\botlc}{\ensuremath{0}_C\xspace}
\newcommand{\andc}{\ensuremath{\cap}_C\xspace}
\newcommand{\orc}{\ensuremath{\cup}_C\xspace}
\newcommand{\andi}{\ensuremath{\cap}_I\xspace}
\newcommand{\ori}{\ensuremath{\cup}_I\xspace}
\newcommand{\andlc}{\ensuremath{\sqcap}_C\xspace}
\newcommand{\orlc}{\ensuremath{\sqcup}_C\xspace}
\newcommand{\andli}{\ensuremath{\sqcap}_I\xspace}
\newcommand{\orli}{\ensuremath{\sqcup}_I\xspace}
\newcommand{\ANDC}{\ensuremath{\:\hat{\cap}_C\:}\xspace}
\newcommand{\ORC}{\ensuremath{\:\check{\cup}_C\:}\xspace}
\newcommand{\ANDI}{\ensuremath{\:\hat{\cap}_I\:}\xspace}
\newcommand{\ORI}{\ensuremath{\:\check{\cup}_I\:}\xspace}

\newcommand{\TOPI}{\ensuremath{\:\hat{\mathrm{t}}_I\:}\xspace}
\newcommand{\BOTI}{\ensuremath{\:\check{\mathrm{f}}_I\:}\xspace}
\newcommand{\TOPC}{\ensuremath{\:\hat{\mathrm{t}}_C\:}\xspace}
\newcommand{\BOTC}{\ensuremath{\:\check{\mathrm{f}}_C\:}\xspace}

\newcommand{\ANDLC}{\ensuremath{\:\hat{\sqcap}_C\:}\xspace}
\newcommand{\ORLC}{\ensuremath{\:\check{\sqcup}_C\:}\xspace}
\newcommand{\ANDLI}{\ensuremath{\:\hat{\sqcap}_I\:}\xspace}
\newcommand{\ORLI}{\ensuremath{\:\check{\sqcup}_I\:}\xspace}

\newcommand{\TOPLI}{\ensuremath{\:\hat{1}_I\:}\xspace}
\newcommand{\BOTLI}{\ensuremath{\:\check{0}_I\:}\xspace}
\newcommand{\TOPLC}{\ensuremath{\:\hat{1}_C\:}\xspace}
\newcommand{\BOTLC}{\ensuremath{\:\check{0}_C\:}\xspace}

\newcommand{\wcircc}{\ensuremath{\circ}_C\,\xspace}
\newcommand{\wcirci}{\ensuremath{\circ}_I\,\xspace}
\newcommand{\bcircc}{\ensuremath{{\bullet_C}}\,\xspace}
\newcommand{\bcirci}{\ensuremath{{\bullet_I}}\,\xspace}
\newcommand{\bboxi}{\ensuremath{\blacksquare}_I\,\xspace}
\newcommand{\bboxc}{\ensuremath{\blacksquare}_C\,\xspace}
\newcommand{\bdiai}{\ensuremath{\Diamondblack}_I\,\xspace}
\newcommand{\bdiac}{\ensuremath{\Diamondblack}_C\,\xspace}
\newcommand{\wboxi}{\ensuremath{\Box}_I\,\xspace}
\newcommand{\wboxc}{\ensuremath{\Box}_C\,\xspace}
\newcommand{\wdiai}{\ensuremath{\Diamond}_I\,\xspace}
\newcommand{\wdiac}{\ensuremath{\Diamond}_C\,\xspace}
\newcommand{\bboxl}{\ensuremath{\Diamondblack}_I\ \xspace}

\newcommand{\WCIRCC}{\ensuremath{\tilde{\circ}_C\,}\xspace}
\newcommand{\WCIRCI}{\ensuremath{\tilde{\circ}_I\,}\xspace}
\newcommand{\WCIRC}{\ensuremath{\tilde{\circ}\,}\xspace}
\newcommand{\BCIRCC}{\ensuremath{\:\tilde{{\bullet}}_C\,}\xspace}
\newcommand{\BCIRCI}{\ensuremath{\:\tilde{{\bullet}}_I\,}\xspace}
\newcommand{\BCIRC}{\ensuremath{\:\tilde{{\bullet}}\,}\xspace}
\newcommand{\WBOXI}{\ensuremath{\:\check{\Box}_I\,}\xspace}
\newcommand{\WBOXC}{\ensuremath{\:\check{\Box}_C\,}\xspace}
\newcommand{\WBOX}{\ensuremath{\:\check{\Box}\,}\xspace}
\newcommand{\BBOXI}{\ensuremath{\:\check{\blacksquare}_I\,}\xspace}
\newcommand{\BBOXC}{\ensuremath{\:\check{\blacksquare}_C\,}\xspace}

\newcommand{\BDIAI}{\ensuremath{\:\hat{\Diamondblack}_I\,}\xspace}
\newcommand{\BDIAC}{\ensuremath{\:\hat{\Diamondblack}_C\,}\xspace}
\newcommand{\BDIA}{\ensuremath{\:\hat{\Diamondblack}\,}\xspace}
\newcommand{\WDIAI}{\ensuremath{\:\hat{\Diamond}_I\,}\xspace}
\newcommand{\WDIAC}{\ensuremath{\:\hat{\Diamond}_C\,}\xspace}

\begin{document}
\title{Logics for Rough Concept Analysis\thanks{The research of the fourth author is supported by the NWO Vidi grant 016.138.314, the NWO Aspasia grant 015.008.054, and a Delft
Technology Fellowship awarded to the fourth author in 2013.}}
\author{Giuseppe Greco\inst{1}\orcidID{0000-0002-4845-3821} \and
Peter Jipsen\inst{2} \and
Krishna Manoorkar\inst{3} \and
Alessandra Palmigiano\inst{4,5}\orcidID{0000-0001-9656-7527}\and
Apostolos Tzimoulis\inst{2}
}
\authorrunning{Greco, Jipsen, Manoorkar, Palmigiano, Tzimoulis}
\institute{Utrecht University, NL \and
Chapman University, US \and
Indian Institute of Technology, Kanpur, India \and
Delft University of Technology, NL \and
Department of Pure and Applied Mathematics, University of Johannesburg, SA
}
\maketitle              % typeset the header of the contribution
\begin{abstract}
Taking an algebraic perspective on the basic structures of Rough Concept Analysis as the starting point, in this paper we introduce some varieties of lattices expanded with normal modal operators which can be  regarded as the natural rough algebra counterparts of certain subclasses of rough formal contexts, and introduce proper display calculi for the logics associated with these varieties which are sound, complete,  conservative and with uniform cut elimination and subformula property. These calculi modularly extend the multi-type calculi for rough algebras to a `nondistributive' (i.e.~general lattice-based) setting.

\end{abstract}

\section{Introduction}

This paper continues a line of investigation started in \cite{GLMP18} and aimed at introducing sequent calculi for the logics of varieties of `rough algebras', introduced and discussed in \cite{banerjee1996rough,saha2014algebraic}. The `rough algebras' considered in  the present paper are {\em nondistributive} (i.e.~general lattice-based) generalizations of those of \cite{saha2014algebraic}; specifically, they are  varieties of lattices expanded with normal modal operators, natural examples of which arise in connection with (certain subclasses of) {\em rough formal contexts}, introduced by Kent in \cite{kent1996rough} as the basic notion of {\em Rough Concept Analysis} (RCA),  a synthesis of Rough Set Theory  \cite{pawlak1998rough} and Formal Concept Analysis \cite{ganter2012formal}. The core idea of Kent's approach is to use a given indiscernibility relation $E$ on the objects of a formal context $(A, X, I)$ to generate  $E$-definable approximations  $R$ and $S$ of the relation $I$ such that $S\subseteq I\subseteq R$. The starting point of our approach is that  $R$ and $S$ can be used  to generate tuples of adjoint normal  modal operators $\langle S\rangle \dashv [S]$ and $\langle R\rangle \dashv [R]$. We identify conditions under which $[S]$ and $\langle R\rangle$ are interior operators and  $[R]$ and $\langle S\rangle$ are closure operators. This provides the basic algebraic framework, which we axiomatically extend so as to define `nondistributive' counterparts of the varieties introduced in \cite{saha2014algebraic}.

From an algebraic perspective, it is interesting to observe that, unlike $\langle S\rangle$ and $[S]$, the modal operators $\langle R\rangle$ and $[R]$ play the reverse  roles they usually have in rough set theory: namely, $[R]$, being an {\em inflationary} map, plays naturally the role of the {\em closure} operator  providing the upper lax approximation of a given formal concept, and similarly $\langle R\rangle$, being a {\em deflationary} map, plays the role of the {\em interior} operator, providing the lower lax approximation of a given formal concept.

From a proof-theoretic perspective, these properties make it possible to introduce a modular generalization of the {\em multi-type approach} taken in  \cite{GLMP18} to endow the logics of `rough algebras' with analytic  calculi, so as to adapt it to a  `nondistributive' propositional base.
For the sake of introducing the structural counterparts of the lattice connectives $\wedge$ and $\vee$ (the reasons for which are explained below), our basic calculus does not have the display property, since the usual display rules for $\wedge$ and $\vee$ are not sound in the general lattice setting. However, the cut elimination and subformula property for the calculi defined in Section \ref{ssec:Display calculus} can be straightforwardly verified by appealing to the meta-theorem of \cite{TrendsXIII}.  Another interesting departure from the calculi of \cite{GLMP18} concerns the counterparts of the IA3 condition, which in the present paper comes in two variants: the lower (strict), and the upper (lax).
The inequality corresponding to the lower variant of IA3, which was analytic in the presence of distributivity, is not analytic inductive in the absence of distributivity (cf.~\cite[Definition 55]{greco2016unified}). However, the inequality corresponding to the upper variant of IA3 is analytic inductive, and hence can be captured in terms of an analytic structural rule. 

\section{Preliminaries}
\label{ssec:prelim FCA}
The purpose of this section, which is based   on \cite[Appendix]{TarkPaper} and \cite{conradie2016categories} and \cite[Sections 2.3 and 2.4]{KriPalNac18}, is to briefly recall the basic notions of the theory of  {\em enriched formal contexts} (cf.~Definition \ref{def:enriched formal context}) while introducing the notation which will be used throughout the paper.
For any relation $T\subseteq U\times V$, and any $U'\subseteq U$  and $V'\subseteq V$, let
\[T^{(0)}[V']:=\{u\mid \forall v(v\in V'\Rightarrow uTv) \}   \quad\quad  T^{(1)}[U']:=\{v\mid \forall u(u\in U'\Rightarrow uTv) \}.\]
It can be easily verified that $U'\subseteq T^{(0)}[V']$ iff $V'\subseteq T^{(1)}[U']$, that  $V_1\subseteq V_2\subseteq V$ (resp.~$U_1\subseteq U_2\subseteq U$) implies that $T^{(0)}[V_2]\subseteq T^{(0)}[V_1]$ (resp.~$T^{(1)}[U_2]\subseteq T^{(1)}[U_1]$), and  $S\subseteq T\subseteq U\times V$ implies that $S^{(0)}[V']\subseteq T^{(0)}[V']$ and $S^{(1)}[U']\subseteq T^{(1)}[U']$ for all $V'\subseteq V$ and $U'\subseteq U$.

  {\em Formal contexts}, or {\em polarities},  are structures $\mathbb{P} = (A, X, I)$ such that $A$ and $X$ are sets, and $I\subseteq A\times X$ is a binary relation. Intuitively, formal contexts can be understood as abstract representations of databases \cite{ganter2012formal}, so that  $A$ represents a collection of {\em objects}, $X$ as a collection of {\em features}, and for any object $a$ and feature $x$, the tuple $(a, x)$ belongs to $I$ exactly when object $a$ has feature $x$. In what follows, we use $a, b$ (resp.~$x, y$) for elements of $A$ (resp.~$X$), and $B$ (resp.~$Y$) for subsets of $A$ (resp.~of $X$).

 As is well known, for every formal context $\mathbb{P} = (A, X, I)$, the pair of maps \begin{center}$(\cdot)^\uparrow: \mathcal{P}(A)\to \mathcal{P}(X)\quad \mbox{ and } \quad(\cdot)^\downarrow: \mathcal{P}(X)\to \mathcal{P}(A),$\end{center}
respectively defined by the assignments $B^\uparrow: = I^{(1)}[B]$ and $Y^\downarrow: = I^{(0)}[Y]$,  form a Galois connection % (cf.~Lemma \ref{lemma: basic}.2), \marginnote{here we refer to Lemma 1.2 and 1.3. Check the definition of $\uparrow$ and $\downarrow$: source / target.}
 and hence induce the closure operators $(\cdot)^{\uparrow\downarrow}$ and $(\cdot)^{\downarrow\uparrow}$ on $\mathcal{P}(A)$ and on $\mathcal{P}(X)$ respectively.\footnote{When $B=\{a\}$ (resp.\ $Y=\{x\}$) we write $a^{\uparrow\downarrow}$ for $\{a\}^{\uparrow\downarrow}$ (resp.~$x^{\downarrow\uparrow}$ for $\{x\}^{\downarrow\uparrow}$).} Moreover, the fixed points of these closure operators form complete sub-$\bigcap$-semilattices  of $\mathcal{P}(A)$ and $\mathcal{P}(X)$  respectively, and hence are complete lattices which are dually isomorphic to each other via the restrictions of the maps $(\cdot)^{\uparrow}$ and $(\cdot)^{\downarrow}$. % (cf.~Lemma \ref{lemma: basic}.3).
 This motivates the following
\begin{definition}
For every formal context $\mathbb{P} = (A, X, I)$, a {\em formal concept} of $\mathbb{P}$ is a pair $c = (B, Y)$ such that $B\subseteq A$, $Y\subseteq X$, and $B^{\uparrow} = Y$ and $Y^{\downarrow} = B$.  The set $B$ is the {\em extension} of  $c$, which we will sometimes denote $\val{c}$, and $Y$ is the {\em intension} of $c$, sometimes denoted $\descr{c}$. Let $L(\mathbb{P})$ denote the set of the formal concepts of $\mathbb{P}$. Then the {\em concept lattice} of $\mathbb{P}$ is the complete lattice  \begin{center}$\mathbb{P}^+: = (L(\mathbb{P}), \bigwedge, \bigvee),$\end{center} where for every $\mathcal{X}\subseteq L(\mathbb{P})$, \begin{center}$\bigwedge \mathcal{X}: = (\bigcap_{c\in \mathcal{X}} \val{c}, (\bigcap_{c\in \mathcal{X}} \val{c})^{\uparrow})\quad \mbox{ and }\quad \bigvee \mathcal{X}: = ((\bigcap_{c\in \mathcal{X}} \descr{c})^{\downarrow}, \bigcap_{c\in \mathcal{X}} \descr{c}). $\end{center}
Then clearly, $\top^{\mathbb{P}^+}: = \bigwedge\varnothing = (A, A^{\uparrow})$ and $\bot^{\mathbb{P}^+}: = \bigvee\varnothing = (X^{\downarrow}, X)$, and the partial order underlying this lattice structure is defined as follows: for any $c, d\in L(\mathbb{P})$, \begin{center}$c\leq d\quad \mbox{ iff }\quad \val{c}\subseteq \val{d} \quad \mbox{ iff }\quad \descr{d}\subseteq \descr{c}.$\end{center}
\end{definition}
\begin{theorem} \label{thm:Birkhoff} (Birkhoff's theorem, main theorem of FCA) Any complete lattice $\mathbb{L}$ is isomorphic to the concept lattice $\mathbb{P}^+$ of some formal context $\mathbb{P}$.
\end{theorem}

\begin{definition} \label{def:enriched formal context}
An {\em enriched formal context} is a tuple
	$\mathbb{F} = (\mathbb{P}, R_\Box, R_\Diamond)$
	such that $\mathbb{P} = (A, X, I)$ is a formal context, and $R_\Box\subseteq A\times X$ and $R_\Diamond \subseteq X\times A$  are $I$-{\em compatible} relations, that is, \label{def:I-compatible rel}  $R_{\Box}^{(0)}[x]$ (resp.~$R_{\Diamond}^{(0)}[a]$) % $R^{\downarrow}[x^{\downarrow\uparrow}]$
	and $R_{\Box}^{(1)}[a]$ (resp.~$R_{\Diamond}^{(1)}[x]$)  %$R^{\uparrow}[a^{\uparrow\downarrow}]$
	are Galois-stable for all $x\in X$ and $a\in A$.
	The {\em complex algebra} of $\mathbb{F}$ is
	\[\mathbb{F}^+ = (\mathbb{P}^+, [R_\Box], \langle R_\Diamond\rangle),\]
	where $\mathbb{P}^+$ is the concept lattice of $\mathbb{P}$, and $[R_\Box]$ and $\langle R_\Diamond\rangle$ are unary operations on $\mathbb{P}^+$ defined as follows: for every $c \in \mathbb{P}^+$,
	\[[R_\Box]c: = (R_{\Box}^{(0)}[\descr{c}], (R_{\Box}^{(0)}[\descr{c}])^{\uparrow}) \quad \mbox{ and }\quad \langle R_\Diamond\rangle c: = ((R_{\Diamond}^{(0)}[\val{c}])^{\downarrow}, R_{\Diamond}^{(0)}[\val{c}]).\]
\end{definition}
 Since $R_{\Box}$ and $ R_{\Diamond}$ are $I$-compatible,  $[R_\Box],\langle R_{\Diamond}\rangle, [R_\Diamond^{-1}],\langle R_{\Box}^{-1}\rangle: \mathbb{P}^+\to \mathbb{P}^+$ are well-defined. %\marginnote{(cf.~\cite[Lemma ???]{rough concepts})}
\begin{lemma} (cf.~\cite[Lemma 3]{KriPalNac18})
	For any  enriched formal context $\mathbb{F} = (\mathbb{P}, R_{\Box}, R_{\Diamond})$, the algebra $\mathbb{F}^+ = (\mathbb{P}^+, [R_{\Box}], \langle R_{\Diamond}\rangle)$ is a complete  lattice expanded with normal modal operators such that $[R_\Box]$ is completely meet-preserving and $\langle R_\Diamond\rangle$ is completely join-preserving.
\end{lemma}

\begin{definition}
\label{def:relational composition}
	For any formal context $\mathbb{P}= (A,X,I)$ and any $I$-compatible relations $R, T\subseteq A\times X$, the {\em composition} $R\, ;T\subseteq A \times X$  is defined as follows: for any $a\in A$ and  $x \in X$,
		\[(R\, ;T)^{(1)} [a] = R^{(1)}[I^{(0)}[T^{(1)} [a]]] ~or~equivalently~ (R\, ;T)^{(0)} [x] = R^{(0)}[I^{(1)}[T^{(0)} [x]]].  \]	
\end{definition} 

\section{Motivation: Kent's Rough Concept Analysis}
Below, we report on the basic definitions and constructions in Rough Concept Analysis \cite{kent1996rough}, cast in the notational conventions of Section \ref{ssec:prelim FCA}.

{\em Rough formal contexts} (abbreviated as {\em Rfc}) are tuples $\mathbb{G} = (\mathbb{P}, E)$ such that $\mathbb{P} = (A, X, I)$ is a polarity (cf.~Section \ref{ssec:prelim FCA}), and $E\subseteq A\times A$ is an equivalence relation (the {\em indiscernibility} relation between objects). For every $a\in A$ we let $(a)_E: = \{b\in A\mid aEb\}$.
The relation $E$ induces two  relations $R, S\subseteq A\times I$  approximating $I$, defined as follows: for every $a\in A$ and $x\in X$,
\begin{equation}\label{eq:lax approx}
aRx \, \mbox{ iff }\, bIx \mbox{ for some } b\in (a)_E;
\quad\quad\quad\quad
aSx \, \mbox{ iff }\, bIx \mbox{ for all } b\in (a)_E.
%S^{(0)}[x]: = \{a\in A\mid x\in I^{(1)[(a)_E]}\} = \{a\in A\mid \forall b(aEb\Rightarrow bIx)\}= \{a\in A\mid (a)_E \subseteq I^{(0)}[x]\}.
\end{equation}
By definition,  $R, S$ are $E$-{\em definable} (i.e.~$R^{(0)}[x] = \bigcup_{aRx}(a)_E$ and $S^{(0)}[x]= \bigcup_{aSx}(a)_E$ for any $x\in X$), and $E$ being reflexive immediately implies that
\begin{lemma}
\label{lemma:strict is reflexive, lax is subdelta}
For any Rfc $\mathbb{G} = (\mathbb{P}, E)$, if $R$ and $S$ are defined as in \eqref{eq:lax approx}, then
\begin{equation}
\label{eq:kbs and reflexivity-subdelta}
S\subseteq I \quad\textrm{and}\quad I\subseteq R.\end{equation}
\end{lemma}
Intuitively, we can think of $R$ as the {\em lax} version of $I$ determined by $E$, and $S$ as its {\em strict} version determined by $E$.
Following the methodology introduced in \cite{CoPa:non-dist} and applied in \cite{conradie2016categories,TarkPaper} to introduce a polarity-based semantics for the modal logics of formal concepts, under the assumption that $R$ and $S$ are $I$-compatible (cf.~Definition \ref{def:enriched formal context}), the relations $R$ and $S$ can be used to define normal modal operators $[R], \langle R\rangle, [S], \langle S\rangle$ on $\mathbb{P}^+$ defined as follows: for any $c\in \mathbb{P}^+$,
\begin{equation}\label{eq:lax box}
\val{[R]c} : =R^{(0)}[\descr{c}] = \{a\in A\mid \forall x(x\in \descr{c}\Rightarrow aRx)\}
\end{equation}
\begin{equation}\label{eq:strict box}
\val{[S]c}: =S^{(0)}[\descr{c}] = \{a\in A\mid \forall x(x\in \descr{c}\Rightarrow aSx)\}.
 \end{equation}
 That is, the members of $[R]c$  are exactly those objects that satisfy (possibly by proxy of some object equivalent to them) all features in the description of $c$, while the members of $[S]c$  are exactly those objects that not only satisfy all features in the description of $c$, but that `force' all their equivalents to also satisfy them.
The assumption that $S\subseteq I$ implies that $\val{[S] c} = S^{(0)}[\descr{c}] \subseteq  I^{(0)}[\descr{c}]= \val{c}$, hence $[S]c$ is a sub-concept of $c$. The assumption that $I \subseteq R$ implies that $\val{c} = I^{(0)}[\descr{c}]\subseteq R^{(0)}[\descr{c}] = \val{[R] c}$, hence $[R]c$ is a super-concept of $c$.
Moreover, for any $c\in \mathbb{P}^+$,
\begin{equation}\label{eq:lax diamond}
\descr{\langle R\rangle c} : = R^{(1)}[\val{c}] = \{x\in X\mid \forall a(a\in \val{c}\Rightarrow aRx)\}
\end{equation}
\begin{equation}\label{eq:strict diamond}
\descr{\langle S\rangle c} : = S^{(1)}[\val{c}] = \{x\in X\mid \forall a(a\in \val{c}\Rightarrow aSx)\}.
 \end{equation}
 That is, $\langle R\rangle c$ is the concept described by those features  shared not only by  each member of $c$ but also by their equivalents, while $\langle S\rangle c$ is the concept described by the common features of those members of $c$ which `force' each of their equivalents to share them.
The assumption that $I\subseteq R$ implies that $\descr{c} = I^{(1)}[\val{c}]\subseteq R^{(1)}[\val{c}] = \descr{\langle R\rangle c}$, and hence $\langle R\rangle c$ is a sub-concept of $c$.  The assumption that $S\subseteq I$ implies that $\descr{\langle S\rangle c} =  S^{(1)}[\val{c}] \subseteq I^{(1)}[\val{c}] = \descr{c} $, and hence $\langle S\rangle c$ is a super-concept of $c$.
Summing up the discussion above, we have verified that the conditions $I\subseteq R$ and $S\subseteq I$ imply that the following sequents of the modal logic of formal concepts are valid on Kent's basic structures:
\begin{equation}\label{eq:axioms for reflexivity and sub-delta}
\Box_s \phi\vdash \phi\quad  \phi\vdash \Box_{\ell}\phi\quad  \phi\vdash \Diamond_s\phi\quad  \Diamond_{\ell}\phi\vdash \phi,
\end{equation}
where $\Box_s$ is interpreted as $[S]$, $\Box_{\ell}$  as $[R]$, $\Diamond_s$ as $\langle S\rangle$ and $\Diamond_{\ell}$ as $\langle R\rangle$. Translated algebraically, these conditions say that $\Box_s$ and $\Diamond_{\ell}$ are  {\em deflationary}, as {\em interior} operators are, $\Diamond_s$ and $\Box_{\ell}$ are  {\em inflationary}, as {\em closure} operators are. Hence, it is natural to ask under which conditions they (i.e.~their semantic interpretations) are indeed  closure/interior operators. The next definition and lemma provide answers to this question. % in the positive for a natural subclass of Rfc.
\begin{definition}
An  Rfc $\mathbb{G} = (\mathbb{P}, E)$ is {\em amenable} if $E$, $R$ and $S$ (defined as in \eqref{eq:lax approx}) are $I$-compatible.\footnote{The assumption that $E$ is $I$-compatible does not follow from $R$ and $S$ being $I$-compatible. Let $\mathbb{G} = (\mathbb{P}, Id_A)$ for any polarity $\mathbb{P}$ such that not all singleton sets of objects are Galois-stable. Hence $E = Id_A$ is not $I$-compatible. However, if $E = Id_A$, then $R = S = I$ are $I$-compatible.}
\end{definition}

\begin{lemma}
\label{lemma:strict is transitive, lax is dense}
For any amenable Rfc $\mathbb{G} = (\mathbb{P}, E)$, if  and $R$ and $S$ are defined as in \eqref{eq:lax approx}, then
\begin{equation}
\label{eq:kbs and transitivity-denseness}
R;R\subseteq R \quad \mbox{ and }\quad S\subseteq S; S.\end{equation}
\end{lemma}
\begin{proof}
Let $x\in X$. To show that $R^{(0)}[I^{(1)}[R^{(0)}[x]]]\subseteq R^{(0)}[x]$, let $a\in R^{(0)}[I^{(1)}[R^{(0)}[x]]]$. By adjunction, this is equivalent to $I^{(1)}[R^{(0)}[x]]\subseteq R^{(1)}[a]$, which implies that $ I^{(0)}[R^{(1)}[a]]\subseteq I^{(0)}[I^{(1)}[R^{(0)}[x]]] = R^{(0)}[x]$, the last equality holding since $R$ is $I$-compatible by assumption. Moreover, $I\subseteq R$ (cf.~Lemma \ref{lemma:strict is reflexive, lax is subdelta}) implies that $I^{(1)}[a]\subseteq R^{(1)}[a]$, which implies that $ I^{(0)}[R^{(1)}[a]]\subseteq I^{(0)}[I^{(1)}[a]] \subseteq (a)_E$, the last inclusion holding since  $E$ is $I$-compatible by assumption.
Hence, $I^{(0)}[R^{(1)}[a]]\subseteq R^{(0)}[x]\cap (a)_E$. Suppose for contradiction that $a\notin R^{(0)}[x]$. By the $E$-definability of $R$, this is equivalent to $R^{(0)}[x]\cap (a)_E = \varnothing$. Hence $I^{(0)}[R^{(1)}[a]] = \varnothing$, from which it follows that $R^{(1)}[a] = I^{(1)}[I^{(0)}[R^{(1)}[a]]] = I^{(1)}[\varnothing] = X$. Hence, $x\in R^{(1)}[a]$, i.e.~$a\in R^{(0)}[x]$,  against the assumption that $a\notin R^{(0)}[x]$.

Let $x\in X$. To show that $S^{(0)}[x]\subseteq S^{(0)}[I^{(1)}[S^{(0)}[x]]]$, assume that $a\in S^{(0)}[x]$. Since $S$ is $E$-definable by construction, this is equivalent to $(a)_E\subseteq S^{(0)}[x]$. To show that $a\in S^{(0)}[I^{(1)}[S^{(0)}[x]]]$, we need to show that $b Iy$ for any $b\in (a)_E$ and  any $y\in I^{(1)}[S^{(0)}[x]]$. Let  $y\in I^{(1)}[S^{(0)}[x]]$. Hence, by definition,  $b'Iy$ for every $b'\in S^{(0)}[x]$. Since $(a)_E\subseteq S^{(0)}[x]$, this implies that $b Iy$ for any $b\in (a)_E$, as required.
\end{proof}
By the general theory developed in \cite{CoPa:non-dist} and applied  to enriched formal contexts in \cite[Proposition 5]{KriPalNac18},  properties \eqref{eq:kbs and transitivity-denseness} guarantee that  the following sequents of the modal logic of formal concepts are also valid on amenable Rfc's:
\begin{equation}\label{eq:axioms for transitivity and denseness}
\Box_s \phi\vdash \Box_s\Box_s \phi\quad  \Box_{\ell}\Box_{\ell}\phi\vdash \Box_{\ell}\phi\quad  \Diamond_s\Diamond_s\phi\vdash \Diamond_s\phi\quad  \Diamond_{\ell}\phi\vdash \Diamond_{\ell}\Diamond_{\ell}\phi.\end{equation}
Finally, again by \cite[Proposition 5]{KriPalNac18}, the fact that by construction $\Box_s$ and $\Diamond_s$ (resp.~$\Box_{\ell}$ and $\Diamond_{\ell}$) are interpreted by operations defined in terms of the same relation guarantees the validity of the following sequents on amenable Rfc's:
\begin{equation}\label{eq:axioms for symmetry}
\phi\vdash  \Box_s\Diamond_s \phi\quad\quad \Diamond_s\Box_s \phi\vdash  \phi\quad\quad \phi\vdash  \Box_{\ell}\Diamond_{\ell} \phi\quad\quad \Diamond_{\ell}\Box_{\ell} \phi\vdash  \phi.\end{equation}
Axioms \eqref{eq:axioms for reflexivity and sub-delta}, \eqref{eq:axioms for transitivity and denseness} and \eqref{eq:axioms for symmetry}   constitute the starting point and motivation for the proof-theoretic investigation of the logics associated to varieties of algebraic structures which can be understood as abstractions   of amenable Rfc's. We define these varieties in the next section.

\section{Kent algebras}
\label{sec:kent algebras}
In the present section, we introduce {\em basic Kent algebras} (and the variety of {\em abstract Kent algebras} (aKa)  to which they naturally belong),   as  algebraic generalizations of amenable Rfc's, and then introduce some subvarieties of aKas in the style of \cite{saha2014algebraic}.
\begin{definition}
\label{def:basic Kent algebras}
A {\em basic Kent algebra} is a structure $\mathbb{A} = (\mathbb{L}, \Box_s, \Diamond_s, \Box_{\ell}, \Diamond_{\ell})$ such that $\mathbb{L}$ is a complete lattice, and $\Box_s, \Diamond_s, \Box_{\ell}, \Diamond_{\ell}$ are unary operations on $\mathbb{L}$ such that for all $a, b\in \mathbb{L}$,
\begin{equation}
\label{eq:adjunction}
\Diamond_s a\leq b\mbox{ iff } a\leq \Box_s b \quad \mbox{ and }\quad \Diamond_{\ell} a\leq b\mbox{ iff } a\leq \Box_{\ell} b,
\end{equation}
 and for any $a\in\mathbb{L}$,
 \begin{equation}\label{eq:reflexive}\Box_s a\leq  a\quad\quad a\leq \Diamond_s a\quad \quad  a\leq \Box_{\ell} a\quad\quad \Diamond_{\ell} a\leq  a
 \end{equation}
  \begin{equation}\label{eq:transitive}\Box_s a\leq  \Box_s\Box_s a\quad \quad \Diamond_s\Diamond_s a\leq \Diamond_s a\quad\quad \Box_{\ell}\Box_{\ell} a\leq  \Box_{\ell} a\quad \quad \Diamond_{\ell} a\leq \Diamond_{\ell}\Diamond_{\ell} a
  \end{equation}
We let $\mathsf{KA}^+$ denote the class of basic Kent algebras.
\end{definition}
From \eqref{eq:adjunction} it follows that, in basic Kent algebras, $\Box_s$ and $\Box_{\ell}$ are completely meet-preserving, $\Diamond_s$ and  $\Diamond_{\ell}$ are completely join-preserving.
For any amenable Rfc $\mathbb{G} = (\mathbb{P}, E)$, if $R$ and $S$ are defined as in \eqref{eq:lax approx}, then
\[\mathbb{G}^+: = (\mathbb{P}^+, [S], \langle S\rangle, [R], \langle R\rangle)\]
where $\mathbb{P}^+$ is the concept lattice of the formal context $\mathbb{P}$ and $[S], \langle S\rangle, [R], \langle R\rangle$ are defined as in \eqref{eq:lax box}--\eqref{eq:strict diamond}. The following proposition is an immediate consequence of  \cite[Proposition 5]{KriPalNac18}, using Lemmas \ref{lemma:strict is reflexive, lax is subdelta} and \ref{lemma:strict is transitive, lax is dense}, and the fact that $[R]$ and $\langle R\rangle$ (resp.~$[S]$ and $\langle S\rangle$) are defined using the same relation.
\begin{proposition}
If   $\mathbb{G} = (\mathbb{P}, E)$ is an amenable Rfc, then  $\mathbb{G}^+$ is a basic Kent algebra.
\end{proposition}
The natural variety containing basic Kent algebras is defined as follows.
\begin{definition}
\label{def:conceptual tqBa}
An {\em abstract Kent algebra} (aKa) is a structure $\mathbb{A} = (\mathbb{L}, \Box_s, \Diamond_s, \Box_{\ell}, \Diamond_{\ell})$ such that $\mathbb{L}$ is a lattice, and $\Box_s, \Diamond_s, \Box_{\ell}, \Diamond_{\ell}$ are unary operations on $\mathbb{L}$ validating \eqref{eq:adjunction}, \eqref{eq:reflexive} and \eqref{eq:transitive}.
We let $\mathsf{KA}$  %, $\mathsf{TRA}$, $\mathsf{SRA}$, $\mathsf{RTRA}$ etc)
denote the class of abstract Kent algebras. % (resp.~reflexive acras, transitive acras, symmetric acras, reflexive and transitive acras, etc).
\end{definition}
From \eqref{eq:adjunction} it follows that, in aKas, $\Box_s$ and $\Box_{\ell}$ are finitely meet-preserving, $\Diamond_s$ and  $\Diamond_{\ell}$ are finitely join-preserving.

\begin{lemma} %\marginnote{please check that the inequalities are correct}
For any  aKa $\mathbb{A} = (\mathbb{L}, \Box_s, \Diamond_s, \Box_{\ell}, \Diamond_{\ell})$ and every $a\in \mathbb{L}$,
\begin{equation}\label{eq: lower and upper approx}
\Box_s a\vee \Diamond_{\ell} a\leq   \Box_{\ell} a\wedge \Diamond_s a.
\end{equation}
\begin{equation}\label{eq:symmetric} a\leq  \Box_s\Diamond_s a\quad\quad \Diamond_s\Box_s a\leq  a\quad\quad a\leq  \Box_{\ell}\Diamond_{\ell} a\quad\quad \Diamond_{\ell}\Box_{\ell} a\leq  a\end{equation}
\begin{equation}\label{eq:pseudo 5}
\Box_s a\leq \Box_s \Diamond_s a\quad \quad \Diamond_s \Box_s a\leq \Diamond_s a\quad\quad   \Diamond_{\ell} a\leq \Box_{\ell} \Diamond_{\ell} a \quad \quad \Diamond_{\ell} \Box_{\ell} a\leq \Box_{\ell} a.
\end{equation}
\begin{equation}\label{eq:aKa5}
\Diamond_s\Box_s a\leq  \Box_s a\quad \quad \Diamond_s a\leq \Box_s\Diamond_s a\quad\quad \Diamond_{\ell}\Box_{\ell} a\leq  \Box_{\ell} a\quad \quad \Diamond_{\ell} a\leq \Box_{\ell}\Diamond_{\ell} a.
\end{equation}
\end{lemma}
\begin{proof}
The inequalities in \eqref{eq:symmetric}  are  straightforward consequences of \eqref{eq:adjunction}.
The inequalities in \eqref{eq: lower and upper approx} and \eqref{eq:pseudo 5} follow from \eqref{eq:reflexive} and \eqref{eq:symmetric}, using the transitivity of the order.
The inequalities in \eqref{eq:aKa5} follow from those in \eqref{eq:transitive} using \eqref{eq:adjunction}.
\end{proof}
Conditions \eqref{eq:aKa5} define the`Kent algebra' counterparts of topological quasi Boolean algebras 5 (tqBa5) \cite{saha2014algebraic}. In the next definition, we introduce `Kent algebra' counterparts of some other varieties considered in \cite{saha2014algebraic}, %intermediate algebras of types 2 and 3 (IA2 and IA3), and pre-rough algebras (pra) \cite{saha2014algebraic},
and also varieties characterized by interaction axioms  between lax and strict connectives which follow the pattern of the 5-axioms in rough algebras.
\begin{definition}
\label{def:aKa5 etc}
An aKa $\mathbb{A}$ as above is an {\em aKa5'} if for any $a\in\mathbb{L}$,

\begin{equation}\label{eq:aKafive}
\Diamond_{\ell} a \leq \Box_s \Diamond_{\ell} a \quad\quad \Diamond_s \Box_{\ell} a \leq \Box_{\ell} a \quad\quad \Box_s a \leq \Diamond_{\ell} \Box_s a \quad\quad \Box_{\ell} \Diamond_s a \leq \Diamond_s a;
\end{equation}
%is a {\em  K-IA2} if for any $a, b\in\mathbb{L}$,
%\begin{equation}\label{eq:K-IA2 strict}
%\Box_s(a\vee b)\leq  \Box_s a\vee \Box_s b\quad\quad \Diamond_s a\wedge \Diamond_s b\leq \Diamond_s (a\wedge b),
%\end{equation}
%\begin{equation}\label{eq:K-IA2 lax}
%\Box_{\ell}(a\vee b)\leq  \Box_{\ell} a\vee \Box_{\ell} b\quad\quad \Diamond_{\ell} a\wedge \Diamond_{\ell} b\leq \Diamond_{\ell} (a\wedge b),
%\end{equation}
 is  a {\em K-IA3$_s$}  if for any $a, b\in\mathbb{L}$,
\begin{equation}\label{eq:K-IA3 strict}
\Box_s a\leq  \Box_s b \mbox{ and } \Diamond_s a\leq \Diamond_s  b\mbox{ imply } a\leq b,
\end{equation}
and is  a {\em K-IA3$_{\ell}$}  if for any $a, b\in\mathbb{L}$,
\begin{equation}\label{eq:K-IA3 lax}
\Box_{\ell} a\leq  \Box_{\ell} b \mbox{ and } \Diamond_{\ell} a\leq \Diamond_{\ell}  b\mbox{ imply } a\leq b.
\end{equation}
%A {\em K-prerough algebra} (K-pra) is an aKa $\mathbb{A}$ verifying \eqref{eq:aKafive}-\eqref{eq:K-IA3 lax}.
\end{definition}

Interestingly, the third and fourth inequality in \eqref{eq:aKafive} are not analytic inductive (cf.~\cite[Definition 55]{greco2016unified}); however, they are equivalent to analytic inductive inequalities in the multi-type language of the heterogeneous algebras discussed in the next section.

\section{Multi-type presentation of Kent algebras}
\label{sec:multi-type presentation kent}
Similarly to what holds for rough algebras (cf.~\cite[Section 3]{GLMP18}), since the modal operations of any aKa $\mathbb{A} = (\mathbb{L}, \Box_s, \Diamond_s, \Box_{\ell}, \Diamond_{\ell})$ are either interior operators or closure operators, each of them factorizes  into a pair of adjoint normal modal operators which are retractions or co-retractions, as illustrated in the following table:
\begin{center}
\begin{tabular}{|c|c|c|c|}
\hline
    \multicolumn{1}{|c}{$\Box_s = \wcirci\cdot \bboxi$} & \multicolumn{1}{c|}{$\bboxi\cdot \wcirci = id_{\mathrm{S}_I}$} & \multicolumn{1}{|c}{$\Diamond_s = \wcircc\cdot\bdiac$} & \multicolumn{1}{c|}{$\bdiac\cdot \wcircc = id_{\mathrm{S}_C}$} \\
    \hline
    \ \ \ $\wcirci: \mathrm{S}_I\hookrightarrow \mathbb{L}$ \ \ \ & \multicolumn{1}{c|}{$\bboxi: \mathbb{L}\twoheadrightarrow \mathrm{S}_I$} & \multicolumn{1}{|c|}{\ \ \ $\bdiac: \mathbb{L}\twoheadrightarrow \mathrm{S}_C$\ \ \ } & $\wcircc: \mathrm{S}_C\hookrightarrow \mathbb{L}$\\
    \hline \hline
     \multicolumn{1}{|c}{$\Box_{\ell} = \wboxc\cdot \bcircc$} & \multicolumn{1}{c|}{$\bcircc\cdot \wboxc = id_{\mathrm{L}_C}$} & \multicolumn{1}{|c}{$\Diamond_{\ell} = \wdiai \cdot \bcirci$} & \multicolumn{1}{c|}{$\bcirci\cdot \wdiai  = id_{\mathrm{L}_I}$} \\
     \hline
        $\bcircc: \mathbb{L}\twoheadrightarrow \mathrm{L}_C$ &$\wboxc: \mathrm{L}_C\hookrightarrow \mathbb{L}$ & \multicolumn{1}{|c|}{$\wdiai: \mathrm{L}_I\hookrightarrow \mathbb{L}$} &  $\bcirci: \mathbb{L}\twoheadrightarrow \mathrm{L}_I$  \\
    \hline
\end{tabular}
\end{center}
where  $\mathrm{S}_I:= \Box_s[\mathbb{L}]$, $\mathrm{S}_C:= \Diamond_s[\mathbb{L}]$, $\mathrm{L}_C:= \Box_{\ell}[\mathbb{L}]$, and $\mathrm{L}_I:= \Diamond_s[\mathbb{L}]$, and such that for all $\alpha\in \mathrm{S}_I$, $\delta\in \mathrm{S}_C$, $a\in \mathbb{L}$, $\pi\in \mathrm{L}_I$, $\sigma\in \mathrm{L}_C$,
\begin{equation}
\label{eq: multi-type retraction co-retr strict}
\wcirci\alpha\leq a \mbox{ iff } \alpha\leq \bboxi a\quad \bdiac a\leq \delta\mbox{ iff } a\leq \wcircc\delta \quad \bcircc a\leq \pi \mbox{ iff } a\leq \wboxc \pi\quad \wdiai\sigma \leq a\mbox{ iff } \sigma \leq \bcircc a.
\end{equation}
Again similarly to what observed in \cite{GLMP18}, the lattice structure of $\mathbb{L}$ can be exported to each of the sets $\mathrm{S}_I, \mathrm{S}_C, \mathrm{L}_C$ and $\mathrm{L}_I$ via the corresponding pair of modal operators as follows.
\begin{definition}
\label{def:kernel}
For any aKa $\mathbb{A}$, the {\em strict interior kernel} $\mathsf{S_I}= (\mathrm{S}_I, \cup_I, \cap_I, \topi, \boti)$  and the {\em  strict closure kernel} $\mathsf{S_C} = (\mathrm{S}_C, \cup_C, \cap_C,  \topc, \botc)$ are such that, for all $\alpha, \beta\in S_I$, and all $\delta, \gamma\in S_C$,
\begin{center}
\begin{tabular}{lcl}
$\alpha \cup_I \beta: = \bboxi(\wcirci \alpha \lor \wcirci\beta)$ & $\quad\quad\quad$& $\delta \cup_C \gamma: = \bdiac(\wcircc \delta \lor \wcircc\gamma)$\\
$\alpha \cap_I \beta: = \bboxi(\wcirci \alpha \land \wcirci\beta)$ & & $\delta \cap_C \gamma: = \bdiac(\wcircc \delta \land \wcircc\gamma)$\\
$\topi:  = \bboxi\top,\ \boti: = \bboxi\bot$&& $\topc:  = \bdiac\top,\ \botc = \bdiac\bot$\\
\end{tabular}
\end{center}
The {\em  lax interior kernel} $\mathsf{L_I}= (\mathrm{L}_I, \sqcup_I, \sqcap_I,  1_\mathbb{I}, 0_\mathbb{I})$  and the {\em lax closure kernel} $\mathsf{L_C} = (\mathrm{L}_C, \sqcup_C, \sqcap_C,  1_\mathbb{C}, 0_\mathbb{C})$ are such that, for all $\pi, \xi\in L_I$, and all $\sigma, \tau\in L_C$,
\begin{center}
\begin{tabular}{lcl}
$\pi \sqcup_I \xi: = \bcirci(\wdiai \pi \lor \wdiai\xi)$ &$\quad\quad\quad$& $\sigma \sqcup_C \tau: = \bcircc(\wboxc \sigma \lor \wboxc\tau)$\\
$\pi \sqcap_I \xi: = \bcirci(\wdiai \pi \land \wdiai\xi)$ && $\sigma \sqcap_C \tau: = \bcircc(\wboxc \sigma \land \wboxc\tau)$\\
$1_I:  = \bcirci\top,\ 0_I: = \bcirci\bot$ && $1_C:  = \bcircc\top,\ 0_C = \bcircc\bot$\\
\end{tabular}
\end{center}
\end{definition}
Similarly to what observed in \cite{GLMP18}, it is easy to verify that the algebras defined above are lattices,
 and the operations indicated with a circle (either black or white) are lattice homomorphisms (i.e.~are both normal box-type and normal diamond-type operators). The construction above justifies the following definition of  class of heterogeneous algebras %(cf.~\cite{birkhoff lipson})
equivalent to aKas:
\begin{definition}
\label{def: heterogeneous algebras}
A {\em heterogeneous aKa} (haKa)  is a tuple \begin{center}$\mathbb{H} = (\mathbb{L}, \mathsf{S_I},\mathsf{S_C}, \mathsf{L_I},\mathsf{L_C}, \wcirci, \bboxi, \wcircc, \bdiac, \bcirci, \wdiai, \bcircc, \wboxc)$\end{center} such that:
\begin{itemize}
\item[H1] $\mathbb{L}, \mathsf{S_I},\mathsf{S_C}, \mathsf{L_I},\mathsf{L_C}$ are bounded lattices;

\item[H2] $\wcirci: \mathsf{S_I}\hookrightarrow \mathbb{L}$, $\wcircc: \mathsf{S_C}\hookrightarrow \mathbb{L}$, $\bcirci: \mathbb{L}\twoheadrightarrow \mathsf{L_I}$, $\bcircc: \mathbb{L}\twoheadrightarrow \mathsf{L_C}$  are  lattice homomorphisms;
\item[H3] $\wcirci \dashv \bboxi$ \quad\quad $\bdiac \dashv \wcircc$ \quad \quad $\bcircc \dashv \wboxc$ \quad\quad $\wdiai \dashv \bcirci$;
\item[H4]    $\bboxi\wcirci = id_{\mathsf{S_I}} \quad\quad \bdiac\wcircc = id_{\mathsf{S_C}} \quad\quad \bcircc\wboxc = id_{\mathsf{L_C}}\quad\quad \bcirci\wdiai = id_{\mathsf{L_I}}$\footnote{Condition H3 implies that $\bboxi: \mathbb{L}\twoheadrightarrow \mathsf{S_I}$ and $\wboxi: \mathsf{L_I} \hookrightarrow \mathbb{L}$ are $\wedge$-hemimorphisms and $\bdiac: \mathbb{L}\twoheadrightarrow \mathsf{S_C}$ and $\wdiac: \mathsf{L_C} \hookrightarrow \mathbb{L}$ are $\vee$-hemimorphisms; condition  H4 implies that the black connectives are surjective and the white ones are injective.}
\end{itemize}
The haKas corresponding to the varieties of Definition \ref{def:aKa5 etc} are defined as follows: %\marginnote{controlla le 5'}
{\fns
\begin{center}
\begin{tabular}{|c|c|c|}
\hline
Algebra & Acronym & Conditions\\
\hline
{\em heterogeneous aKa5'} & haKa5' & $\wdiai\pi\leq \wcirci\bboxi\wdiai\pi$ $\quad\wcircc\bdiac\wboxc\sigma\leq \wboxc\sigma$\\
&& $\wcirci\alpha\leq \wdiai\bcirci\wcirci\alpha$ $\quad\wboxc\bcircc\wcircc\delta\leq \wcircc\delta$\\
% {\em heterogeneous K-IA2} &hK-IA2 & H6: $\bboxi(a \lor b) \leq \bboxi a \cup_I \bboxi b\quad $ $\bdiac a \cap_C \bdiac b\leq \bdiac(a \land b)$.\\
  % & &  $\;\;\wdiai\pi\land \wdiai \xi\leq \wdiai(\pi \sqcap_C \xi)\quad $  $ \wboxc(\delta \sqcup_C \gamma)\leq \wboxc\delta\lor \wboxc \gamma$.\\
 \hline
 {\em heterogeneous K-IA3$_s$}  &hK-IA3$_s$ & $\bboxi a \leq \bboxi b$ and $ \bdiac a \leq  \bdiac b$ imply $a \leq b$\\
  \hline
 {\em heterogeneous K-IA3$_\ell$}  &hK-IA3$_\ell$ & $\wboxc \bcircc a \leq \wboxc \bcircc b$ and $ \wdiai \bcirci a \leq  \wdiai \bcirci b$ imply $a \leq b$\\
 \hline
% {\em heterogeneous K-pra} &hKpra& H6, H7.\\
% \hline
 \end{tabular}
 \end{center}
 }
 Notice that the inequalities defining  haKa5' are all analytic inductive.
%$$
%\begin{tikzpicture}[node/.style={circle, draw, fill=black}, scale=1]
%\node (DL) at (-1.5,-1.5) {$\mathbb{L}$};
%\node (DM) at (1.5,-1.5) {$\mathbb{D}$};
%\node (adj1) at (0,-1.1) {{\rotatebox[origin=c]{270}{$\vdash$}}};
%\node (adj2) at (0,-2.0) {{\rotatebox[origin=c]{270}{$\vdash$}}};
%\draw [right hook->] (DL)  tonode[below] {$e$}  (DM);
%\draw [->>] (DM)  to[out= 225,in= 315, looseness=1]  node[below] {$\gamma$}  (DL);
%\draw [->>] (DM) to [out=135,in=45, looseness=1]   node[above] {$\iota$}  (DL);
%\end{tikzpicture}
%$$
%In what follows, we use the abbreviated names of the heterogeneous algebras written in ``blackboard bold''  (e.g.~$\mathbb{HTQBA}$, etc.) to indicate their corresponding classes.
A heterogeneous algebra $\mathbb{H}$ is {\em perfect} if every  lattice  in the signature of $\mathbb{H}$ is  perfect (cf.~\cite[Definition 1.8]{CoPa:non-dist}), and every homomorphism (resp.~hemimorphism) in the signature of $\mathbb{H}$ is a complete homomorphism (resp.~hemimorphism).
\end{definition}
Similarly to what discussed in \cite[Section 3]{GLMP18}, one can readily show that the classes of haKas defined above correspond to the varieties defined in Section \ref{sec:kent algebras}. That is, for any aKa $\mathbb{A}$ one can define its corresponding haKa $\mathbb{A}^+$ using the factorizations described at the beginning of the present section and Definition \ref{def:kernel}, and conversely, given a haKa $\mathbb{H}$, one can define its corresponding aKa $\mathbb{H}_+$ by endowing its first domain $\mathbb{L}$ with modal operations defined by taking the appropriate compositions of pairs of heterogeneous maps of $\mathbb{H}$. Then, for every $\mathbb{K}\in \{$aKa, aKa5', K-IA3$_s$, K-IA3$_\ell\}$, letting $\mathbb{HK}$ denote its corresponding class of heterogeneous algebras, the following holds:
\begin{proposition}
\label{prop:from single to multi}
\begin{enumerate}
\item If $\mathbb{A} \in \mathbb{K}$, then    $\mathbb{A}^+\in \mathbb{HK}$;
\item If $\mathbb{H}\in \mathbb{HK}$, then $\mathbb{H}_+\in \mathbb{K}$;
\item $ \mathbb{A} \cong (\mathbb{A}^+)_+ \quad \mbox{and}\quad \mathbb{H} \cong (\mathbb{H}_+)^+.$
\item The isomorphisms of the previous item restrict to perfect members of $\mathbb{K}$ and $\mathbb{HK}$.
\item If $\mathbb{A} \in \mathbb{K}$, then $\mathbb{A}^\delta\cong ((\mathbb{A}^+)^\delta)_+$ and if $\mathbb{H} \in \mathbb{HK}$, then $\mathbb{H}^\delta\cong ((\mathbb{H}_+)^\delta)^+$.
\end{enumerate}
\end{proposition}

\section{Multi-type calculi for the logics of Kent algebras}
\label{ssec:Display calculus}
%\label{ssec:rules}
%\label{ssec:language of DtqBa}
In the present section, we introduce the  multi-type calculi associated with each class of algebras $\mathsf{K}\in \{aKa, aKa5', K\text{-}IA3_{\ell}\}$. The language of these logics matches the language of haKas, and  is built up from structural and operational (i.e.~logical) connectives.  Each structural connective is denoted by decorating its corresponding logical connective  with $\hat{\phantom{a}}$ (resp.~$\check{\phantom{a}}$ or $\tilde{\phantom{a}}$).  %The  will be denote which belongs only to the  family  $\mathcal{F}$ (resp.~$\mathcal{G}$) defined in  \cite{greco2017multi} is denoted by
In what follows, we will adopt the convention that unary connectives bind more strongly than binary ones.

\begin{center}
{\fns
\begin{tabular}{ll}

%\hline
\mc{2}{c}{general lattice $\mathsf{L}$} \\
\mc{2}{c}{
\!\!\!\!\!\!\!\,$A  ::= \,p \mid \xtop \mid \xbot \mid \wcirci \alpha \mid \wcircc \delta \mid \wdiai \pi \mid \wboxc \sigma \mid A \land A \mid A \lor A$
}
 \\
\mc{2}{c}{
$X ::= \,A \mid  \XBOT \mid \XTOP \mid \WCIRCI \Gamma \mid \WCIRCC \Delta \mid \WDIAI \Pi \mid \WBOXI \Pi \mid \WDIAC \Sigma \mid \WBOXC \Sigma \mid X \XAND X \mid X \XOR X$
}
 \\

\mc{2}{c}{\ } \\
%\hline

strict-interior kernel \ $\mathsf{S_I}$ \ & \ lax-interior kernel \ $\mathsf{L_I}$ \\
$\alpha ::= \,\bdiai A \mid \bboxi A$ \ & \ $\pi ::= \,\bcirci A$ \\

$\Gamma ::= \alpha \mid \BDIAI X \mid \BBOXI X \mid \BOTI \mid \TOPI \mid \Gamma \ANDI \Gamma \mid \Gamma \ORI \Gamma$ \ & \ $\Pi ::= \pi \mid \BCIRCI X \mid \BOTLI \mid \TOPLI \mid \Pi \ANDLI \Pi \mid \Pi \ORLI \Pi$ \\

 & \\
%\hline

strict-closure kernel \ $\mathsf{S_C}$ \ &\ lax-closure kernel \ $\mathsf{L_C}$ \\
$\delta ::= \,\bdiac A \mid \bboxc A$ \ & \ $\sigma ::= \,\bcircc A$ \\

$\Delta ::=\, \delta \mid \BDIAC X \mid \BBOXC X \mid \BOTC \mid \TOPC \mid \Delta \ANDC \Delta \mid \Delta \ORC \Delta$ \ & \ $\Sigma ::=\, \sigma \mid \BCIRCC X \mid \BOTLC \mid \TOPLC \mid \Sigma \ANDLC \Sigma \mid \Sigma \ORLC \Sigma$ \\

\end{tabular}
 }
\end{center}

%%%

\begin{itemize}
\item Interpretation of structural connectives as their logical counterparts\footnote{ The connectives  which  appear in a grey cell in the synoptic tables will   only be included in the present language at the structural level.}
\end{itemize}

\begin{enumerate}
\item structural and operational pure $\mathsf{L}$-type connectives:
{
\begin{center}
\begin{tabular}{|r|c|c|c|c|c|}
\hline
\ structural operations \ &$\XTOP$ &$\XBOT$  & $\XAND$& $\XOR$ \\
\hline
\ logical operations \ &$\xtop$ & $\xbot$ & $\xand$ & $\xor$ \\
\hline
\end{tabular}
\end{center}
}
\item structural and operational pure $\mathsf{S_I}$-type and $\mathsf{S_C}$-type connectives:
{
\begin{center}
\begin{tabular}{|r|c|c|c|c|c|c|c|c|c|c|}
\hline
\ structural operations \ &\rule[0mm]{0mm}{0.32cm}$\TOPI$ &$\BOTI$ & $\ANDI$& $\ORI$ & $\TOPC$ & $\BOTC$ & $\ANDC$ & $\ORC$ \\
\hline
\ logical operations \ &\cellcolor{gray!25}$\topi$ & \cellcolor{gray!25}$\boti$ & \cellcolor{gray!25}$\andi$ & \cellcolor{gray!25}$\ori$ & \cellcolor{gray!25}$\topc$ & \cellcolor{gray!25}$\botc$ & \cellcolor{gray!25}$\orc$ & \cellcolor{gray!25}$\andc$ \\
\hline
\end{tabular}
\end{center}
}
\item structural and operational pure $\mathsf{L_I}$-type and $\mathsf{L_C}$-type connectives:
{
\begin{center}
\begin{tabular}{|r|c|c|c|c|c|c|c|c|c|c|}
\hline
\ structural operations \ &\rule[0mm]{0mm}{0.35cm}$\TOPLI$ &$\BOTLI$ & $\ANDLI$& $\ORLI$ & $\TOPLC$ & $\BOTLC$ & $\ANDLC$ & $\ORLC$ \\
\hline
\ logical operations \ &\cellcolor{gray!25}$\topli$ & \cellcolor{gray!25}$\botli$ & \cellcolor{gray!25}$\andli$ & \cellcolor{gray!25}$\orli$ & \cellcolor{gray!25}$\toplc$ & \cellcolor{gray!25}$\botlc$ & \cellcolor{gray!25}$\orlc$ & \cellcolor{gray!25}$\andlc$ \\
\hline
\end{tabular}
\end{center}
}
\item structural and operational multi-type strict connectives:
{
\begin{center}
\begin{tabular}{|r|c|c|c|c|c|c|c|c|c|}
\hline
\ types \ & \mc{2}{c|}{$\mathsf{L} \rightarrow \mathsf{S_I}$} &\mc{2}{c|}{$\mathsf{L} \rightarrow \mathsf{S_C}$} &$\mathsf{S_I} \rightarrow \mathsf{L}$&$\mathsf{S_C} \rightarrow \mathsf{L}$ \\
\hline
\ structural operations \ &\rule[0mm]{0mm}{0.32cm}$\BDIAI$& $\BBOXI$ &$\BDIAC$ &\BBOXC &$\WCIRCI$& $\WCIRCC$\\
\hline
\ logical operations \ &\cellcolor{gray!25}$\bdiai$&$\bboxi$ & $\bdiac$ &\cellcolor{gray!25}$\bboxc$& $\wcirci$ & $\wcircc$\\
\hline
%\ algebraic counterparts \ &$\iota'$&$\iota^\pi$ & $\gamma^\sigma$ &$\gamma'$& $e^\delta_I$ & $e^\delta_C$\\
\end{tabular}
\end{center}
}
\item structural and operational multi-type lax connectives:
{
\begin{center}
\begin{tabular}{|r|c|c|c|c|c|c|c|c|c|}
\hline
\ types \ & \mc{2}{c|}{$\mathsf{L_I} \rightarrow \mathsf{L}$} &\mc{2}{c|}{$\mathsf{L_C} \rightarrow \mathsf{L}$} &$\mathsf{L} \rightarrow \mathsf{L_I}$&$\mathsf{L} \rightarrow \mathsf{L_C}$ \\
\hline
\ structural operations \ &\rule[0mm]{0mm}{0.32cm}$\WDIAI$& $\WBOXI$ &$\WDIAC$ &\WBOXC &$\BCIRCI$& $\BCIRCC$\\
\hline
\ logical operations \ &$\wdiai$&\cellcolor{gray!25}$\wboxi$ & \cellcolor{gray!25}$\wdiac$ &$\wboxc$& $\bcirci$ & $\bcircc$\\
\hline
%\ algebraic counterparts \ &l$\iota'$&l$\iota^\pi$ & l$\gamma^\sigma$ &l$\gamma'$& l$e^\delta_I$ & l$e^\delta_C$\\
\end{tabular}
\end{center}
}
\end{enumerate}

%%%

%In what follows, we will use $X, Y, Z, W, X'\ldots$ as variables for $\mathsf{L}$-structures, $\Gamma, \Gamma', \ldots$ as variables for $\mathsf{S}_I$-structures, and $\Delta, \Delta', \ldots$ as variables for $\mathsf{S}_C$-structures, $\Pi, \Pi', \ldots$ as variables for $\mathsf{L}_I$-structures, and $\Sigma, \Sigma', \ldots$ as variables for $\mathsf{L}_C$-structures.

In what follows, we will use $x, y, z$ as structural variables of arbitrary types, $a, b, c$ as term variables of arbitrary types.

\begin{comment}
, and the notation $x \,[y]$ to mean that the structure $y$ occurs at a certain depth in the generation tree of the well-formed (and therefore well-typed) context $x$, where $x$ and $y$ might be of different types. Since we need to express interactions between modalities and lattice connectives, we have included the structural counterparts of lattice connectives. However, in the present nondistributive setting, their correspondent display rules are not sound and cannot be included. Hence, the calculi below lack the display property. However, they retain a form of visibility (cf.~\cite{TrendsXIII}) thanks to which, cut elimination follows from \cite[Theorem 4.1]{TrendsXIII}. Specifically, some introduction rules (namely, the {\em translation} rules, see below) are given a deep-calculus presentation, while others (the {\em tonicity} rules) are given a display-calculus presentation which guarantees visibility.  This choice prevents the derivation of lattice-distributivity while at the same time allows to express the interaction between modal operators and lattice connectives in terms of analytic structural rules.
\end{comment}

The calculus $\mathrm{D.AKA}$ consists of the following axiom and rules.
\begin{itemize}
\item Identity and Cut:
{\fns
\begin{center}
\begin{tabular}{cc}
\AXC{$\phantom{x \fCenter y \,[a]}$}
\LL{\scriptsize $Id_\mathsf{L}$}
\UI $p \fCenter p$
\DP
\ \ & \ \
\AX$x \fCenter a$
\AX$a \fCenter y$
\RL{\scriptsize Cut}
\BI$x \fCenter y$
\DP
 \\
\end{tabular}
\end{center}
}
\item Multi-type display rules (we omit the display rules capturing the adjunctions \mbox{$\wdiai \dashv \bcirci \dashv \wboxi$} and \mbox{$\wdiai \dashv \bcirci \dashv \wboxi$}):
{\fns
\begin{center}
\begin{tabular}{cccc}
\AX$\WCIRCI \Gamma  \fCenter X$
\LL{\scriptsize $ad_\mathsf{LS_I}$}
\doubleLine
\UI$\Gamma \fCenter \BBOXI X$
\DP
\ & \
\AX$X \rule[0mm]{0mm}{0.26cm}\fCenter \WCIRCI \Gamma$
\RL{\scriptsize $ad_\mathsf{LS_I}$}
\doubleLine
\UI$\BDIAI X \fCenter \Gamma$
\DP
\ \ & \ \
\AX$X \fCenter \WCIRCC \Delta\rule[0mm]{0mm}{0.26cm}$
\LL{\scriptsize $ad_\mathsf{LS_C}$}
\doubleLine
\UI$\BDIAC X \fCenter \Delta$
\DP
\ &\
\AX$\WCIRCC X \fCenter \Delta$
\RightLabel{\scriptsize $ad_\mathsf{LS_C}$}
\doubleLine
\UI$X \fCenter \BBOXC \Delta$
\DP
 \\
\end{tabular}
\end{center}
}

\begin{comment}
\item Pure-type structural rules: these include standard weakening (w), contraction (c), commutativity (e) and associativity (a) in each type. Below we provide the rules for the general lattice $\mathsf{L}$ (see Appendix for the obvious corresponding rules in all the other types).

{\fns
\begin{center}
\begin{tabular}{rlrl}
\AX$x \,[X] \fCenter y$
\LL{\scriptsize w}
\UI$x \,[X \XAND Y] \fCenter y$
\DP
\ & \
\AX$x \fCenter y \,[X]$
\RL{\scriptsize w}
\UI$x \fCenter y \,[X \XOR Y]$
\DP
\ & \
\AX$x \,[X \XAND X] \fCenter y$
\LL{\scriptsize c}
\UI$x \,[X] \fCenter y$
\DP
\ & \
\AX$x \fCenter y \,[X \XOR X]$
\RL{\scriptsize c}
\UI$x \fCenter y \,[X]$
\DP
 \\

 & & & \\

\AX$x \,[X \XAND Y] \fCenter y$
\LL{\scriptsize e}
\UI$x \,[Y \XAND X] \fCenter y$
\DP
\ & \
\AX$x \fCenter y \,[X \XOR Y]$
\RL{\scriptsize e}
\UI$x \fCenter y \,[Y \XOR X]$
\DP
\ \ & \ \
\AX$x \,[(X \XAND Y) \XAND Z] \fCenter y$
\LL{\scriptsize a}
\doubleLine
\UI$x \,[X \XAND (Y \XAND Z)] \fCenter y$
\DP
\ & \
\AX$x \fCenter y \,[(X \XOR Y) \XOR Z]$
\RL{\scriptsize a}
\doubleLine
\UI$x \fCenter y \,[X \XOR (Y \XOR Z)]$
\DP
 \\
\end{tabular}
\end{center}
}
\end{comment}

\item Multi-type structural rules for strict-kernel operators:%\footnote{Notice that in the case of zeroary connectives, whenever the subscripts $I$ or $C$ are not specified, the rule is valid for both of them, e.g.~the invertible rule $\WCIRC \TOPIC$ is an abbreviation for the invertible rules $\WCIRCI \TOPI$ and $\WCIRCC \TOPC$, and the invertible rule $\BCIRC \BOTL$ is an abbreviation for the invertible rules $\BCIRCI \BOTLI$ and $\BCIRCC \BOTLC$}

{\fns
\begin{center}
\begin{tabular}{@{}rlrl}

\AX$\WCIRCI \TOPI \fCenter X$
\LL{\scriptsize $\WCIRC \TOPI$}
\doubleLine
\UI$\XTOP \fCenter X$
\DP
 &
\AX$X \fCenter \WCIRCI \BOTI$
\RL{\scriptsize $\WCIRCI \BOTI$}
\doubleLine
\UI$X \fCenter \XBOT$
\DP

\ & \

\AX$\WCIRCC \TOPC \fCenter X$
\LL{\scriptsize $\WCIRCC \TOPC$}
\doubleLine
\UI$\XTOP \fCenter X$
\DP
 &
\AX$X \fCenter \WCIRCC \BOTC$
\RL{\scriptsize $\WCIRCC \BOTC$}
\doubleLine
\UI$X \fCenter \XBOT$
\DP
 \\

 & & & \\

\AX$\BDIAI \WCIRCI \Gamma \fCenter \Gamma'$
\LeftLabel{\scriptsize $\BDIAI \WCIRCI$}
\doubleLine
\UI$\Gamma \fCenter \Gamma'$
\DP
 &
\AX$\Gamma' \fCenter \BBOXI \WCIRCI \Gamma$
\RL{\scriptsize $\BBOXI \WCIRCI$}
\doubleLine
\UI$\Gamma' \fCenter \Gamma$
\DP
\ &\
\AX$\BDIAC \WCIRCC \Delta \fCenter \Delta'$
\LL{\scriptsize $\BDIAC \WCIRCC$}
\doubleLine
\UI$\Delta \fCenter \Delta'$
\DP
 &
\AX$\Delta' \fCenter \BBOXC \WCIRCC \Delta$
\RL{\scriptsize $\BBOXC \WCIRCC$}
\doubleLine
\UI$\Delta' \fCenter \Delta$
\DP
 \\

 & & & \\

\AX$\WCIRCI \BDIAI X \fCenter Y$
\LL{\scriptsize $\WCIRCI \BDIAI$}
\UI$X \fCenter Y$
\DP
 &
\AX$Y \fCenter \WCIRCI \BBOXI X$
\RL{\scriptsize $\WCIRCI \BBOXI$}
\UI$Y \fCenter X$
\DP
\ &\
\AX$\WCIRCC \BDIAC X \fCenter Y$
\LL{\scriptsize $\WCIRCC \BDIAC$}
\UI$X \fCenter Y$
\DP
 &
\AX$Y \fCenter \WCIRCC \BBOXC X$
\RL{\scriptsize $\WCIRCC \BBOXC$}
\UI$Y \fCenter X$
\DP
 \\
\end{tabular}
\end{center}
}

\item Multi-type structural rules for lax-kernel operators:

{\fns
\begin{center}
\begin{tabular}{@{}rlrl}

\AX$\BCIRCI \XTOP \fCenter \Pi$
\LL{\scriptsize $\BCIRCI \TOPLI$}
\doubleLine
\UI$\TOPLI \fCenter \Pi$
\DP
 &
\AX$\Pi \fCenter \BCIRCI \XBOT$
\RL{\scriptsize $\BCIRCI \BOTLI$}
\doubleLine
\UI$\Pi \fCenter \BOTLI$
\DP

\ &\

\AX$\BCIRCC \XTOP \fCenter \Sigma$
\LL{\scriptsize $\BCIRC \TOPLC$}
\doubleLine
\UI$\TOPLC \fCenter \Sigma$
\DP
 &
\AX$\Sigma \fCenter \BCIRCC \XBOT$
\RL{\scriptsize $\BCIRCC \BOTLC$}
\doubleLine
\UI$\Sigma \fCenter \BOTLC$
\DP
 \\

 & & & \\

\AX$\Pi \fCenter \Pi'$
\LeftLabel{\scriptsize $\WDIAI \BCIRCI$}
\UI$\WDIAI \BCIRCI \Pi \fCenter \Pi'$
\DP
 &
\AX$\Pi' \fCenter \Pi$
\RL{\scriptsize $\WBOXI \BCIRCI$}
\UI$\Pi' \fCenter \WBOXI \BCIRCI \Pi$
\DP

\ &\

\AX$\Sigma \fCenter \Sigma'$
\LeftLabel{\scriptsize $\WDIAC \BCIRCC$}
\UI$\WDIAC \BCIRCC \Sigma \fCenter \Sigma'$
\DP
 &
\AX$\Sigma' \fCenter \Sigma$
\RL{\scriptsize $\WBOXC \BCIRCC$}
\UI$\Sigma' \fCenter \WBOXC \BCIRCC \Sigma$
\DP

 \\

 & & & \\

\AX$\BCIRCI \WDIAI \Pi \fCenter \Pi'$
\LeftLabel{\scriptsize $\BCIRCI \WDIAI$}
\doubleLine
\UI$\Pi \fCenter \Pi'$
\DP
 &
\AX$\Pi' \fCenter \BCIRCI \WBOXI \Pi$
\RL{\scriptsize $\BCIRCI \WBOXI$}
\doubleLine
\UI$\Pi' \fCenter \Pi$
\DP

\ & \

\AX$\BCIRCC \WDIAC \Sigma \fCenter \Sigma'$
\LeftLabel{\scriptsize $\BCIRCC \WDIAC$}
\doubleLine
\UI$\Sigma \fCenter \Sigma'$
\DP
 &
\AX$\Sigma' \fCenter \BCIRCC \WBOXC \Sigma$
\RL{\scriptsize $\BCIRCC \WBOXC$}
\doubleLine
\UI$\Sigma' \fCenter \Sigma$
\DP
 \\
\end{tabular}
\end{center}
}

\item Multi-type structural rules for the correspondence between kernels:
%\marginnote{ma queste non sono derivabili dalle altre? No, non mi pare: come fai a passare da I a C con le altre?}

{\fns
\begin{center}
\begin{tabular}{rl}
\AX$\WCIRCI\BDIAI X \fCenter Y$
\LL{\scriptsize $\WCIRC \BDIA$}
\doubleLine
\UI$\WCIRCC\BDIAC X \fCenter Y$
\DP
\ \ & \ \
\AX$Y \fCenter \WBOXI\BCIRCI X$
\RL{\scriptsize $\WBOX \BCIRC$}
\doubleLine
\UI$Y \fCenter \WBOXC\BCIRCC X$
\DP
 \\
\end{tabular}
\end{center}
}

%In what follows the unary rules that turn a structural connective $S$ in the premise into its logical counterpart $s$ in the conclusion are given in a deep presentation. We call these rules \textsl{translation rules} (or vertical rules). We call the other introduction rules per each connective \textsl{tonicity rules} (or horizontal rules).

\item Logical rules for multi-type connectives related to strict kernels:
{\fns
\begin{center}
\begin{tabular}{@{}rlrl}
\AX$\BDIAI A \fCenter \Gamma$
\LL{\scriptsize $\bboxl$}
\UI$\bboxl A \fCenter \Gamma$
\DP
\ & \
\AX$\rule[0mm]{0mm}{0.424cm}X \fCenter A$
\RL{\scriptsize $\bdiai$}
\UI$\BDIAI X \fCenter \bboxl A$
\DP
\ \ & \ \
\AX$A \fCenter X\rule[0mm]{0mm}{0.39cm}$
\LL{\scriptsize $\bboxc$}
\UI$\bboxc A \fCenter \BBOXC X$
\DP
\ & \
\AX$\Delta \fCenter \BBOXC A\rule[0mm]{0mm}{0.3cm}$
\RL{\scriptsize $\bboxc$}
\UI$\Delta \fCenter \bboxc A$
\DP
 \\

 & & & \\

\AX$\WCIRCI \alpha \fCenter X$
\LL{\scriptsize $\wcirci$}
\UI$\wcirci \alpha \fCenter X$
\DP
\ & \
\AX$X \fCenter \WCIRCI \alpha$
\RL{\scriptsize $\wcirci$}
\UI$X \fCenter \wcirci \alpha$
\DP

\ \ & \ \

\AX$\WCIRCC \delta \fCenter X$
\LL{\scriptsize $\wcircc$}
\UI$\wcircc \delta \fCenter X$
\DP
\ & \
\AX$X \fCenter \WCIRCC \delta$
\RL{\scriptsize $\wcircc$}
\UI$X \fCenter \wcircc \delta$
\DP
 \\
\end{tabular}
\end{center}
}

\item Logical rules for multi-type connectives related to lax kernels:
{\fns
\begin{center}
\begin{tabular}{rlrl}

\AX$\WDIAI \pi \fCenter X$
\LL{\scriptsize $\wdiai$}
\UI$\wdiai \pi \fCenter X$
\DP
\ & \
\AX$\Pi \rule[0mm]{0mm}{0.424cm}\fCenter \pi$
\RL{\scriptsize $\wdiai$}
\UI$\WDIAI \Pi \fCenter \wdiai \pi$
\DP

\ \ & \ \

\AX$\sigma \fCenter \Sigma\rule[0mm]{0mm}{0.39cm}$
\LL{\scriptsize $\wboxc$}
\UI$\wboxc \sigma \fCenter \WBOXC \Sigma$
\DP
\ & \
\AX$X \fCenter \WBOXC \sigma\rule[0mm]{0mm}{0.3cm}$
\RL{\scriptsize $\wboxc$}
\UI$X \fCenter \wboxc \sigma$
\DP
 \\

 & & & \\

\AX$\BCIRCI A \fCenter \Pi$
\LL{\scriptsize $\wcirci$}
\UI$\bcirci A \fCenter \Pi$
\DP
\ & \
\AX$\Pi \fCenter \BCIRCI A$
\RL{\scriptsize $\bcirci$}
\UI$\Pi \fCenter \bcirci A$
\DP

\ \ & \ \

\AX$\BCIRCC A \fCenter \Sigma$
\LL{\scriptsize $\bcircc$}
\UI$\bcircc A \fCenter \Sigma$
\DP
\ & \
\AX$\Sigma \fCenter \BCIRCC A$
\RL{\scriptsize $\bcircc$}
\UI$\Sigma \fCenter \bcircc A$
\DP
 \\
\end{tabular}
\end{center}
}

\item Logical rules for lattice connectives:

{\fns
\begin{center}
\begin{tabular}{rlrl}

\AX$\XTOP \fCenter X$
\LL{\scriptsize $\xtop$}
\UI$\xtop \fCenter X$
\DP
\ & \
\AXC{$\phantom{\XTOP \fCenter}\rule[0mm]{0mm}{0.38cm}$}
\RL{\scriptsize $\xtop$}
\UI$\XTOP \fCenter \xtop$
\DP

\ \ & \ \

\AXC{$\phantom{[\XTOP] \fCenter}\rule[0mm]{0mm}{0.34cm}$}
\LL{\scriptsize $\xbot$}
\UI$\xbot \fCenter \XBOT$
\DP
\ & \
\AX$X \fCenter \XBOT$
\RL{\scriptsize $\xbot$}
\UI$X \fCenter \xbot$
\DP
 \\

 & & & \\

\AX$A_{i\in\{1,2\}} \fCenter X$
\LL{\scriptsize $\xand_i$}
\UI$A_1 \xand A_2 \fCenter X$
\DP
\ & \
\AX$X \rule[0mm]{0mm}{0.38cm}\fCenter A$
\AX$X \fCenter B$
\RL{\scriptsize $\xand$}
\BI$X \fCenter A \xand B$
\DP

\ \ & \ \

\AX$A \fCenter X$
\AX$B \fCenter X\rule[0mm]{0mm}{0.38cm}$
\LL{\scriptsize $\xor$}
\BI$A \xor B \fCenter X$
\DP
\ & \
\AX$X \fCenter A_{i\in\{1,2\}}$
\RL{\scriptsize $\xor_i$}
\UI$X \fCenter A_1 \xor A_2$
\DP
 \\

\end{tabular}
\end{center}
}

\end{itemize}

The proper display calculi for the subvarieties of $\mathrm{aKa}$ discussed in Section  \ref{sec:kent algebras} are obtained by adding the following rules:

{\fns
\begin{center}
\begin{tabular}{|c|c|c|}
\hline
\ Logic \ & \ Calculus \ & \mc{1}{c|}{Rules}\\

\hline
\ $\mathrm{H.aKa5'}$\
 &
\ $\mathrm{D.aKa5'}$\
 &
\rule[-10mm]{0mm}{2.13cm}
\begin{tabular}{rl}
\AX$\WDIAI \Pi \fCenter X$
\LL{\scriptsize $\WCIRCI \BDIAI \WDIAI$}
\UI$\WCIRCI \BDIAI \WDIAI \Pi \fCenter X$
\DP
\ & \
\AX$X \fCenter \WBOXC \Sigma$
\RL{\scriptsize $\WCIRCC \BBOXC \WBOXC$}
\UI$X \fCenter \WCIRCC \BBOXC \WBOXC \Sigma$
\DP
 \\

% & \\

\AX$\WDIAI \BCIRCI \WCIRCI \Gamma \fCenter X$
\LL{\scriptsize $\WDIAI \BCIRCI \WCIRCI$}
\UI$\WCIRCI \Gamma \fCenter X$
\DP
\ & \
\AX$X \fCenter \WBOXC \BCIRCC \WCIRCC \Delta$
\RL{\scriptsize $\WBOXC \BCIRCC \WCIRCC$}
\UI$X \fCenter \WCIRCC \Delta$
\DP
\rule[0mm]{0mm}{0.7cm}
\
 \\
\end{tabular}
 \\
\hline

%\ $\mathrm{K\textrm{-}IA2}$\
% &
%\ $\mathrm{D.K\textrm{-}IA2}$\
% &
%{\begin{tabular}{rl}
%\rule[0mm]{0mm}{0.65cm}
%\AX$x \,[\BDIAC (X \XAND Y)] \fCenter y$
%\LL{\scriptsize $\BDIAC \ANDC$}
%\UI$x \,[\BDIAC X \ANDC \BDIAC Y] \fCenter y$
%\DP
%\ & \
%\AX$x \fCenter y \,[\BBOXI (X \XOR Y)]$
%\RL{\scriptsize $\BBOXI \ORI$}
%\UI$x \fCenter y \,[\BBOXI X \ORI \BBOXI Y]$
%\DP
%\  \\
%\AX$x \,[\WCIRCC (\Delta \ANDC \Delta')] \fCenter y$
%\LL{\scriptsize $\WCIRCC \ANDC$}
%\UI$x \,[\WCIRCC \Delta \XAND \WCIRCC \Delta'] \fCenter y$
%\DP
%\ & \
%\AX$x \fCenter y \,[\WCIRCI (\Gamma \ORI \Gamma')]$
%\RL{\scriptsize $\WCIRCI \ORI$}
%\UI$x \fCenter y \,[\WCIRCI \Gamma \XOR \WCIRCI \Gamma']$
%\DP
%\rule[0mm]{0mm}{0.7cm}
%\  \\
%
%\rule[0mm]{0mm}{0.7cm}
%\AX$x \,[\BCIRCI (X \XAND Y)] \fCenter y$
%\LL{\scriptsize $\BCIRCI \ANDLI$}
%\UI$x \,[\BCIRCI X \ANDLI \BCIRCI Y] \fCenter y$
%\DP
%\ & \
%\AX$x \fCenter y \,[\BCIRCC (X \XOR Y)]$
%\RL{\scriptsize $\BCIRCC \ORLC$}
%\UI$x \fCenter y \,[\BCIRCC X \ORLC \BCIRCC Y]$
%\DP
%\  \\
%\AX$x \,[\WDIAI (\Pi \ANDLI \Pi')] \fCenter y$
%\LL{\scriptsize $\WDIAI \ANDLI$}
%\UI$x \,[\WDIAI \Pi \XAND \WDIAC \Pi'] \fCenter y$
%\DP
%\ & \
%\AX$x \fCenter y \,[\WBOXC (\Sigma \ORLC \Sigma')]$
%\RL{\scriptsize $\WBOXC \ORLC$}
%\UI$x \fCenter y \,[\WBOXC \Sigma \XOR \WBOXC \Sigma']$
%\DP
%\rule[0mm]{0mm}{0.7cm}
%\rule[-5mm]{0mm}{1cm}
%\  \\

%\end{tabular}}

% \\

%\hline

\ $\mathrm{K\textrm{-}IA3}_{\ell}$\
 &
\ $\mathrm{D.K\textrm{-}IA3}_{\ell}$\
 &

\rule[-4mm]{0mm}{1cm}
\AX$X \fCenter \WBOXI \BCIRCI Y$
\AX$\WDIAC \BCIRCC X \fCenter Y$
\RL{\scriptsize $k\textrm{-}ia3_{\ell}$}
\BI$X \fCenter Y$
\DP
 \\

\hline

%$\mathrm{K\textrm{-}pra}$
% &
%$\mathrm{D.K\textrm{-}pra}$
% &
%\rule[-2.04mm]{0mm}{0.56cm}
%all the previous rules
% \\
%\hline

\end{tabular}
\end{center}
}

These calculi enjoy the properties of soundness, completeness, conservativity, cut elimination and subformula property the verification of which is  standard and follows from the general theory of proper display calculi (cf.~\cite{linearlogPdisplayed,greco2017multi,GrecoPalmigianoLatticeLogic,bilattice,fei2018,apostolos2018,Inquisitive}). These verifications are discussed in the appendix.

\bibliography{ref}
\bibliographystyle{plain}

\newpage
\appendix

\section{Properties}
Throughout this section, we let $\mathsf{K}\in \{aKa, aKa5', %K\text{-}IA2,
K\text{-}IA3_{\ell}\}$, and $\mathsf{HK}$ the class of heterogeneous algebras corresponding to $\mathsf{K}$. Further, we let  %$\mathrm{H.K}$ denote the lattice-based modal logic in the modal signature $\Box_s, \Diamond_s, \Box_{\ell}, \Diamond_{\ell}$ corresponding to $\mathsf{K}$, and $\mathrm{M.K}$ ; let $\mathbb{A}$  and $\mathbb{HA}$ denote its corresponding class of single-type and heterogeneous algebras, respectively, and let
$\mathrm{D.K}$ denote the multi-type calculus for the logic $\mathrm{H.K}$ canonically associated with $\mathsf{K}$.
\label{ssec: properties}
\subsection{Soundness for perfect $\mathsf{HK}$ algebras}
\label{ssec:soundness}
%In the present subsection, we outline
The  verification of the soundness of the rules of $\mathrm{D.K} $ w.r.t.~the semantics of {\em perfect} elements of $\mathsf{HK}$ (see Definition \ref{def: heterogeneous algebras}) is analogous to that of many other multi-type calculi (cf.~\cite{linearlogPdisplayed,greco2017multi,GrecoPalmigianoLatticeLogic,bilattice,fei2018,apostolos2018,Inquisitive}). %The first step consists in interpreting structural symbols as logical symbols according to their (precedent or succedent) position, as indicated at the beginning of Section \ref{ssec:Display calculus}. This makes it possible to interpret sequents as inequalities, and rules as quasi-inequalities. For example, the rule on the left-hand side below is interpreted as the quasi-inequalities on the right-hand side:
Here we only discuss the soundness of the rule $k\textrm{-}ia3_{\ell}$.
By definition, the following quasi-inequality is valid on every  K-IA3$_{\ell}$:
\[\Box_{\ell} a\leq  \Box_{\ell} b \mbox{ and } \Diamond_{\ell} a\leq \Diamond_{\ell}  b\mbox{ imply } a\leq b.\]
This quasi-inequality  equivalently translates into the multi-type language as follows:
\[\wboxc\bcircc a\leq  \wboxc\bcircc b \mbox{ and } \wdiai\bcirci a\leq \wdiai\bcirci  b\mbox{ imply } a\leq b.\]
By adjunction, the quasi-inequality above can be equivalently rewritten as follows:
\[\wdiac\bcircc\wboxc\bcircc a\leq  b \mbox{ and }  a\leq \wboxi\bcirci\wdiai\bcirci  b\mbox{ imply } a\leq b,\]
which, thanks to a well known property of adjoint maps, simplifies as:
\[\wdiac\bcircc a\leq  b \mbox{ and }  a\leq \wboxi\bcirci  b\mbox{ imply } a\leq b.\]
Hence, %$b = \wdiac\bcircc a\vee  b$ and $a = a\wedge \wboxi\bcirci  b$
the quasi-inequality above is equivalent to the following inequality:
\[a\wedge \wboxi\bcirci  b\leq \wdiac\bcircc a\vee  b.\]
The inequality above is analytic inductive (cf.~\cite[Definition 55]{greco2016unified}), and hence running ALBA on this inequality produces:
\begin{center}
\begin{tabular}{r l l}
   & $\forall a\forall b  [ a\wedge \wboxi\bcirci  b\leq \wdiac\bcircc a\vee  b]$ &\\
iff & $\forall p\forall q\forall a\forall b    [ (p\leq a\wedge \wboxi\bcirci  b \ \& \ \wdiac\bcircc a\vee  b\leq q)\Rightarrow p \leq q]$ &\\
iff & $\forall p\forall q\forall a\forall b  [ (p\leq a \ \& \ p\leq \wboxi\bcirci  b \ \&\ b\leq q \ \&\ \wdiac\bcircc a\leq q)\Rightarrow p\leq q]$ &\\
iff & $\forall p\forall q [ (p\leq \wboxi\bcirci  q \ \&\ \wdiac\bcircc p\leq q)\Rightarrow p\leq q]$. &\\
\end{tabular}
\end{center}
The last quasi-inequality above is the semantic translation of the rule $k\textrm{-}ia3_{\ell}$:
\begin{center}
\AX$X \fCenter \WBOXI \BCIRCI Y$
\AX$\WDIAC \BCIRCC X \fCenter Y$
\RL{\scriptsize $k\textrm{-}ia3_{\ell}$}
\BI$X \fCenter Y$
\DP
\end{center}
which we then proved to be sound on every perfect heterogeneous  K-IA3$_{\ell}$, by the soundness of the ALBA steps.
Likewise, the defining condition of K-IA3$_{\ell}$ translates into the inequality \[a\xand \wcircc\bdiac b\leq \wcirci\bboxi a\xor b,\] which, however, is not analytic inductive, and hence it cannot be transformed into an analytic rule via ALBA. %(see Section \ref{ssec: completeness}).
\subsection{Completeness}\label{ssec: completeness}
 Let  $A^\tau\vdash B^\tau$ be the translation of any sequent $A\vdash B$ in the language of $\mathrm{H.K}$  into the language of $\mathrm{D.K}$ induced by the correspondence between $\mathsf{K}$ and $\mathsf{HK}$ described in Section \ref{sec:multi-type presentation kent}. % which composes the  translation introduced in Section \ref{sec: multi-type language} with the correspondence between algebraic operations  and logical connectives indicated in table (iv) of Section \ref{ssec:language of DtqBa}.
%The translations of the axioms and rules of \mathrm{H} are derivable in D. %Since H.SDM is complete w.r.t.~SM-algebras, by Propositions \ref{prop:Aplus plus} and \ref{prop:consequence preserved and reflected}  this is enough to show that D.SDM derives all $\modelsHSM.
%In order to translate sequents of H.SDM into sequents in the language of D.SDM, we will make use of the translations $(\cdot)^t, (\cdot)_t: \mathcal{L}\to \mathcal{L}_{\mathrm{MT}}$, defined simultaneously as follows:

\begin{proposition}
For every $\mathrm{H.K}$-derivable sequent $A \fCenter B$, the sequent $A^{\tau} \fCenter B^{\tau}$ is derivable in $\mathrm{D.K}$.
\end{proposition}

Below we provide the multi-type translations of the single-type sequents corresponding to inequalities \eqref{eq:adjunction}. All of them are derivable in D.AKA by logical introduction rules, display rules, and the rules $\WBOX \BCIRC$ and $\WCIRC \BDIA$.

%Adjunction
\begin{center}
\begin{tabular}{@{}lcl@{}}
$\Diamond_s A \vdash B\mbox{\ \ iff\ \ } A \vdash \Box_s B$ & $\quad\rightsquigarrow\quad$ & $\wcircc \bdiac A \vdash B\mbox{\ \ iff\ \ } A \vdash \wcirci \bboxi B$ \\
$\Diamond_{\ell} A \vdash B \mbox{\ \ iff\ \ } A \vdash \Box_{\ell} B$ & $\quad\rightsquigarrow\quad$ & $\wdiai \bcirci A \vdash B \mbox{\ \ iff\ \ } A \vdash \wboxc \bcircc B$ \\
\end{tabular}
\end{center}

Below we provide the multi-type translations of the single-type sequents corresponding corresponding to inequalities \eqref{eq:reflexive} and \eqref{eq:transitive}, respectively. All of them are derivable in D.AKA by logical introduction rules and display rules.

%Reflexivity & %Transitivity
\begin{center}
\begin{tabular}{@{}lclclcl@{}}
$\Box_s A \vdash A$ & $\quad\rightsquigarrow\quad$ & $\wcirci \bboxi A \vdash A$
 & \ \ \ \ \ \ \ \ \ \ &
$\Box_s A \vdash \Box_s\Box_s A$ & $\quad\rightsquigarrow\quad$ & $\wcirci \bboxi A \vdash \wcirci \bboxi\wcirci \bboxi A$ \\

$A \vdash \Diamond_s A$ & $\quad\rightsquigarrow\quad$ & $A \vdash \wcircc \bdiac A$
 & &
$\Diamond_s \Diamond_s A \vdash \Diamond_s A$ & $\rightsquigarrow$ & $\wcircc \bdiac \wcircc \bdiac A \vdash \wcircc \bdiac A$ \\

$A \vdash \Box_{\ell} A$ & $\quad\rightsquigarrow\quad$ & $A \vdash \wboxc \bcircc A$
 & &
$\Box_{\ell}\Box_{\ell} A \vdash \Box_{\ell} A$ & $\rightsquigarrow$ & $\wboxc \bcircc \wboxc \bcircc A \vdash \wboxc \bcircc A$ \\

$\Diamond_{\ell} A \vdash A$ & $\quad\rightsquigarrow\quad$ & $\wdiai \bcirci A \vdash A$
 & &
$\Diamond_{\ell} A \vdash \Diamond_{\ell}\Diamond_{\ell} A$ & $\rightsquigarrow$ & $\wdiai \bcirci A \vdash \wdiai \bcirci \wdiai \bcirci A$ \\
\end{tabular}
\end{center}

Below we provide the multi-type translation of the single-type sequents corresponding to inequalities \eqref{eq:aKafive}. All of them are derivable in D.AKA5'. %We provide a derivation of the first one.\marginnote{non c'e', ma possiamo anche non metterla, nel qual caso leviamo l'ultima frase}

\begin{center}
\begin{tabular}{@{}lcl@{}}
$\Diamond_{\ell} A \vdash \Box_s \Diamond_{\ell} A$ & $\quad\rightsquigarrow\quad$ & $\wdiai \bcirci A \vdash \wcirci \bboxi \wdiai \bcirci A$ \\
$\Diamond_s \Box_{\ell} A \vdash \Box_{\ell} A$ & $\rightsquigarrow$ & $\wcircc \bdiac \wboxc \bcircc A \vdash \wboxc \bcircc A$ \\
$\Box_s A \vdash \Diamond_{\ell} \Box_s A$ & $\rightsquigarrow$ & $\wcirci \bboxi A \vdash \wdiai \bcirci \wcirci \bboxi A$ \\
$\Box_{\ell} \Diamond_s A \vdash \Diamond_s A$ & $\rightsquigarrow$ & $\wboxc \bcircc \wcircc \bdiac A \vdash \wcircc \bdiac A$ \\
\end{tabular}
\end{center}
Below we provide the multi-type translations of the single-type rules corresponding to quasi-inequality %\eqref{eq:K-IA3 strict} and 
\eqref{eq:K-IA3 lax}, respectively.

%\begin{center}
%$\Diamond_s A\fCenter \Diamond_s B$ \ and \ $\Box_s A \fCenter \Box_s B$ \ imply \ $A \fCenter B \quad \rightsquigarrow \quad$ \\ $\wcircc \bdiac A \fCenter \wcircc \bdiac B$ \ and \ $\wcirci \bboxi A \fCenter \wcirci \bboxi B$ \ imply \ $A \fCenter B$
%\end{center}

\begin{center}
$\Diamond_{\ell} A\fCenter \Diamond_{\ell} B$ \ and \  $\Box_{\ell} A \fCenter \Box_{\ell} B$ \ imply\  $A \fCenter B \quad \rightsquigarrow \quad$ \\ $\wdiai \bcirci A \fCenter \wdiai \bcirci B$ \ and\  $\wboxc \bcircc A \fCenter \wboxc \bcircc B$ \ imply \  $A \fCenter B$
\end{center}

%As discussed in \cite[Section 4]{GLMP18}, condition \eqref{eq:K-IA3 strict} is equivalent to the inequality
%\[a\xand \wcircc\bdiac b\leq \wcirci\bboxi a\xor b,\]
%which is analytic inductive (cf.~\cite[Definition 55]{greco2016unified}) in the distributive setting of tqBas, but is {\em not} analytic inductive in the present, nondistributive setting. This is why we are not considering it here. 
Below, we derive \eqref{eq:K-IA3 lax}. Firstly, $A \xand \wboxc \bcircc B \fCenter \wdiai \bcirci A \xor B$ is derivable via $k\textrm{-}ia3_\ell$ by means of the following derivation $\mathcal{D}$:

{\fns
\begin{center}
\AX$B \fCenter B$
\UI$B \fCenter \wdiai \bcirci A \XOR B$
\UI$B \fCenter \wdiai \bcirci A \xor B$
%\RL{\fns $k\textrm{-}ia3$}
\UI$\BCIRCC B \fCenter \BCIRCC (\wdiai \bcirci A \xor B)$
\UI$\bcircc B \fCenter \BCIRCC (\wdiai \bcirci A \xor B)$
\UI$\wboxc \bcircc B \fCenter \WBOXC \BCIRCC (\wdiai \bcirci A \xor B)$
\UI$A \XAND \wboxc \bcircc B \fCenter \WBOXC \BCIRCC (\wdiai \bcirci A \xor B)$
\UI$A \xand \wboxc \bcircc B \fCenter \WBOXC \BCIRCC (\wdiai \bcirci A \xor B)$

\AX$A \fCenter A$
\UI$A \XAND \wboxc \bcircc B \fCenter A$
\UI$A \xand \wboxc \bcircc B \fCenter A$
\UI$\BCIRCI (A \xand \wboxc \bcircc B) \fCenter \BCIRCI A$
\UI$\BCIRCI (A \xand \wboxc \bcircc B) \fCenter \bcirci A $
\UI$\WDIAI \BCIRCI (A \xand \wboxc \bcircc B) \fCenter \wdiai \bcirci A$
\UI$\WDIAI \BCIRCI (A \xand \wboxc \bcircc B) \fCenter \wdiai \bcirci A \XOR B$
\UI$\WDIAI \BCIRCI (A \xand \wboxc \bcircc B) \fCenter \wdiai \bcirci A \xor B$
\RL{\fns $k\textrm{-}ia3_{\ell}$}
\BI$A \xand \wboxc \bcircc B \fCenter \wdiai \bcirci A \xor B$
\DP
\end{center}
 }

Assuming $\wdiai \bcirci A \fCenter \wdiai \bcirci B$  and  $\wboxc \bcircc A \fCenter \wboxc \bcircc B$, we derive $A\fCenter B$ via cut as follows:

{\fns
\begin{center}
\AX$A \fCenter A$
\AXC{$\wboxc \bcircc A \fCenter \wboxc \bcircc B$}
\dashedLine
\UIC{\fns $Id_{\mathrm{L}} + \bcircc + adj_{\mathrm{LS_C}} + Cut_{\mathrm{L}}$}
\dashedLine
\UIC{$A \fCenter \wboxc \bcircc B$}
\RL{\fns $\xand$}
\BIC{$A \XAND A \fCenter A \xand \wboxc \bcircc B$}
\LL{\fns $C_\mathsf{L}$}
\UIC{$A \fCenter A \xand \wboxc \bcircc B$}

\AXC{$\mathcal{D}$}
\noLine
\UIC{$A \xand \wboxc \bcircc B \fCenter \wdiai \bcirci A \xor B$}
\AXC{$\wdiai \bcirci A \fCenter \wdiai \bcirci B$}
\dashedLine
\UIC{\fns $Id_{\mathrm{L}} + \bcirci + adj_{\mathrm{LS_I}} + Cut_{\mathrm{L}}$}
\dashedLine
\UIC{$\wdiai \bcirci A \fCenter B$}
\AXC{$B \fCenter B$}
\LL{\fns $\xor$}
\BIC{$\wdiai \bcirci A \xor B \fCenter B \XOR B$}
\RL{\fns $C_\mathsf{L}$}
\UIC{$\wdiai \bcirci A \xor B \fCenter B$}
\RL{\fns $Cut_{\mathsf{L}}$}
\BIC{$A \xand \wboxc \bcircc B \fCenter B$}

\RL{\fns $Cut_{\mathsf{L}}$}
\BIC{$A \fCenter B$}
\DP
\end{center}
}

\subsection{Conservativity}\label{ssec: conservativity}

%\marginnote{This section has to be adjusted to the present setting.}

To argue that $\mathrm{D.K}$ is conservative w.r.t.~$\mathrm{H.K}$, we follow the standard proof strategy discussed in \cite{greco2016unified,GKPLori}. %Let  $\vdash_{\mathrm{H.K}}$ denote the  syntactic consequence relation corresponding to $\mathrm{H.K}$ and $\models_{\mathbb{HK}}$ denote the  semantic consequence relation arising from (perfect) heterogeneous algebras in $\mathbb{HK}$.
We need to show that, for all formulas $A$ and $B$ in the language of $\mathrm{H.K}$,  if $A^\tau \vdash B^\tau$ is a $\mathrm{D.K}$-derivable sequent, then  $A \vdash B$ is derivable in $\mathrm{H.K}$. This claim can be proved using  the following facts: (a) The rules of $\mathrm{D.K}$ are sound w.r.t.~perfect members of $\mathsf{HK}$ (cf.~Section \ref{ssec:soundness});  (b) $\mathrm{H.K}$ is complete w.r.t.~the class of perfect algebras in $\mathsf{K}$;  (c) A perfect element of $\mathsf{K}$ is equivalently presented as  a perfect member of $\mathsf{HK}$ so that the semantic consequence relations arising from each type of structures preserve and reflect the translation. Let $A, B$ be as above. If  $A^\tau \vdash B^\tau$ is $\mathrm{D.K}$-derivable, then by (a),  $\models_{\mathbb{HK}} A^\tau \vdash B^\tau$. By (c), this implies that $\models_{\mathsf{K}}  A\vdash B$, where $\models_{\mathsf{K}}$ denotes the semantic consequence relation arising from the perfect members of class $\mathsf{K}$. By (b), this implies that $A\vdash B$ is derivable in $\mathrm{H.K}$, as required.

\subsection{Cut elimination and subformula property}

\begin{comment}
A sequent $x \fCenter y$ is type-uniform if x and y are of the same type \cite[Definition 2.1]{TrendsXIII}.

\begin{proposition}
Every derivable sequent in D.K-pra is type-uniform.
\end{proposition}

\begin{proof}
By an easy induction on the complexity of derivable sequents in  given that all the rules of $\mathrm{D.K\textrm{-}IA3}_{\ell}$ (including zero-ary rules and cuts) preserve type-uniformity \cite[Proposition 2.2]{TrendsXIII}
\end{proof}
\end{comment}
 Cut elimination and subformula property for each $\mathrm{D.K}$ are obtained by verifying the assumptions of %The calculus is not a proper display calculus because the binary lattice connectives are not residuated, nevertheless the cut elimination and subformula property can be inferred from a meta-theorem proved \cite{TrendsXIII} generalzing the strategy introduced by Belnap for display calculi.
\cite[Theorem 4.1]{TrendsXIII}. All of them except $\mathrm{C}'_8$ are readily satisfied by inspecting the rules. Condition $\mathrm{C}'_8$ requires to check that reduction steps can be performed for every application of  cut in which both cut-formulas are principal, which either remove the original cut altogether or replace it by one or more cuts on formulas of strictly lower complexity.  In what follows, we only show  $\mathrm{C}'_8$ for some heterogeneous connectives. %\marginnote{oppure si dice che sono analoghi a quelli del paper di LSFA e bbasta}

%\paragraph*{Pure $\mathsf{D}$-type connectives:}
%\noindent The cases for $\Ineg \alpha$ and $\Cneg\xi$ of D.TQBA5 and its extensions  are standard and similar to the one above.
%\paragraph*{Multi-type connectives:}
{\fns
\begin{center}
\begin{tabular}{ccc}
\bottomAlignProof
\AXC{\ \ \ $\vdots$ \raisebox{1mm}{$\pi_1$}}
\noLine
\UI$\Gamma \fCenter \BBOXI A$
%\RL{\fns $\bboxi$}
\UI$\Gamma \fCenter \bboxi A$
\AXC{\ \ \ $\vdots$ \raisebox{1mm}{$\pi_2$}}
\noLine
\UI$A \fCenter Y$
%\LL{\fns $\bboxi$}
\UI$\bboxi A \fCenter \BBOXI Y$
%
%\LL{\fns Cut}
\BI$\Gamma \fCenter \BBOXI Y$
\DisplayProof

 & $\ \ \ \ \rightsquigarrow\ \ \ \ \ $ &

\!\!\!\!\!\!\!
\bottomAlignProof
\AXC{\ \ \ $\vdots$ \raisebox{1mm}{$\pi_1$}}
\noLine
\UI$\Gamma \fCenter  \BBOXI A$
\UI$\WCIRCI \Gamma \fCenter A$
\AXC{\ \ \ $\vdots$ \raisebox{1mm}{$\pi_2$}}
\noLine
\UI$A \fCenter Y$
\BI$\WCIRCI \Gamma \fCenter Y$
\UI$ \Gamma \fCenter \BBOXI Y$
\DisplayProof
 \\
\end{tabular}
\end{center}

\begin{center}
\begin{tabular}{ccc}
\bottomAlignProof
\AXC{\ \ \ $\vdots$ \raisebox{1mm}{$\pi_1$}}
\noLine
\UI$X \fCenter \WCIRCI \alpha$
\UI$X \fCenter \wcirci \alpha$
\AXC{\ \ \ $\vdots$ \raisebox{1mm}{$\pi_2$}}
\noLine
\UI$\WCIRCI \alpha  \fCenter Y$
\UI$\wcirci \alpha \fCenter Y$
%\RL{\fns $Cut$}
\BI$X \fCenter Y$
\DisplayProof

 & $\ \ \ \ \ \rightsquigarrow\ \ \ \ \ $ &

\!\!\!\!\!\!\!
\bottomAlignProof
\AXC{\ \ \ $\vdots$ \raisebox{1mm}{$\pi_1$}}
\noLine
\UI$X \fCenter \WCIRCI \alpha$
%\RL{\fns $ad_\mathsf{DK_I}$}
\UI$\BDIAI X \fCenter \alpha$
\AXC{\ \ \ $\vdots$ \raisebox{1mm}{$\pi_2$}}
\noLine
\UI$\WCIRCI \alpha  \fCenter Y$
%\LL{\fns $ad_\mathsf{DK_I}$}
\UI$ \alpha \fCenter \BBOXI Y$

%\RL{\fns $Cut$}
\BI$\BDIAI X \fCenter \BBOXI Y$
%\RL{\fns $\BDIAI\BBOXI$}
\UI$X \fCenter Y$
\DisplayProof
\end{tabular}
\end{center}
}
\noindent The remaining cases   are analogous. 

\end{document}